\documentclass[10pt]{amsart}
\allowdisplaybreaks
\usepackage{fullpage}
\usepackage{amssymb}
\usepackage[all]{xy}
\usepackage{bbm}
\usepackage{cancel}
\usepackage{graphicx}
\usepackage{mathabx}
\usepackage{tikz-cd}
\usepackage{quiver}
\usepackage{epigraph} 
\usepackage{stmaryrd}
\usepackage{enumitem}
\usepackage[mathscr]{eucal}
\usepackage{soul}
\usepackage{mathtools}
\usepackage{xfrac}
\usepackage{verbatim}

\usepackage[utf8]{inputenc}
\usepackage[T1]{fontenc}
\usepackage[backend=biber,
            isbn=false,
            doi=false,
            maxbibnames=9,
            giveninits=true,
            style=alphabetic,
            citestyle=alphabetic]{biblatex}
\usepackage{url}
\renewbibmacro{in:}{}
\bibliography{paper.bib}

\usepackage{xcolor}
\definecolor{e-mail}{rgb}{0,.40,.80}
\definecolor{reference}{rgb}{.20,.60,.22}
\definecolor{citation}{rgb}{0,.40,.80}
\usepackage[colorlinks=true,
            linkcolor=reference,
            citecolor=citation,
            urlcolor=e-mail]{hyperref}

\usepackage{aliascnt}
\usepackage[capitalize]{cleveref}
\usepackage{adjustbox}

\theoremstyle{plain}
\newtheorem*{maintheorem}{Theorem}

\newtheorem{theorem}{Theorem}[section]
\newtheorem*{theorem*}{Theorem}

\newtheorem{corollary}[theorem]{Corollary}
\newtheorem{proposition}[theorem]{Proposition}
\newtheorem{lemma}[theorem]{Lemma}
\newtheorem{conjecture}[theorem]{Conjecture}
\newtheorem*{conjecture*}{Conjecture}
\newaliascnt{assumption}{theorem}

\aliascntresetthe{assumption}
\crefname{assumption}{assumption}{assumptions}
\theoremstyle{definition}
\newtheorem{definition}[theorem]{Definition}
\newtheorem*{definition*}{Definition}
\newtheorem{example}[theorem]{Example}
\newtheorem{remark}[theorem]{Remark}
\newtheorem*{notation*}{Notation and Convention}

\newcommand{\Filt}{\textup{Filt}}
\newcommand{\Grad}{\textup{Grad}}
\newcommand{\X}{\mathcal{X}}
\newcommand{\dX}{\mathfrak{X}}
\newcommand{\Y}{\mathcal{Y}}
\newcommand{\dY}{\mathfrak{Y}}
\newcommand{\M}{\mathcal{M}}

\newcommand{\F}{\mathcal{F}}

\renewcommand{\O}{\mathcal{O}}

\newcommand{\A}{\mathcal{A}}
\renewcommand{\S}{\mathbb{S}}
\renewcommand{\L}{\mathbb{L}}
\DeclareMathOperator{\Hom}{Hom}
\DeclareMathOperator{\Spec}{\mathrm{Spec}}

\title{Motivic DT/PT Correspondence}
\author{Chih-Huan Chang}
\date{\today}

\begin{document}
\address{Institute of Mathematics, Academia Sinica, Taipei, Taiwan}
\email{chchang0501@as.edu.tw}
\begin{abstract}
We prove a motivic version of the Donaldson--Thomas/Pandharipande--Thomas (DT/PT) correspondence on Calabi--Yau threefolds. The proof combines Toda's wall crossing framework and the motivic integral identity recently proved by Bu. This provides a refinement of the numerical DT/PT correspondence.
\end{abstract}

\maketitle

\tableofcontents


\section*{Introduction}
\subsection*{Curve counting theories on Calabi--Yau threefolds}
Enumerative geometry seeks to count geometric objects that satisfy specified 
conditions. In the setting of Calabi--Yau threefolds, various curve-counting 
theories have emerged, each offering different perspectives and techniques.  
Gromov--Witten (GW) theory provides one approach, using stable maps and 
intersection theory on their moduli spaces.  Although we will not use GW theory 
in this paper, it forms part of the broader landscape of curve-counting theories 
and is closely related to Donaldson--Thomas (DT) theory through the MNOP 
conjecture \cite{MNOP2006I, MNOP2006II}.  In this work, however, we focus 
exclusively on DT and Pandharipande--Thomas (PT) theories and their motivic 
refinements.  We now recall the DT and PT theories, which will be the main 
objects of study in this paper.

Let $X$ be a Calabi--Yau threefold, i.e. a smooth projective threefold
over $\mathbb{C}$ with trivial canonical bundle and $H^1(X,\mathcal{O}_X)=0$.
 For a non-negative integer $n$ and a 
homology class $\beta\in H_2(X,\mathbb{Z})$, consider $I_n(X,\beta)$, the moduli 
space of ideal sheaves $\mathcal{I}\subset\mathcal{O}_X$ with Chern character 
$\mathrm{ch}(\mathcal{I}) = (1,0,-\beta,-n)$.  This moduli space carries a 
zero-dimensional virtual fundamental class, and its Donaldson--Thomas invariant is 
defined by
\[
\mathrm{DT}_{n,\beta} := \int_{[I_n(X,\beta)]^{\mathrm{vir}}} 1 \in \mathbb{Z}.
\]
We introduce the generating functions
\[
Z_\beta^{\mathrm{DT}}(q) := \sum_{n\in\mathbb{Z}} \mathrm{DT}_{n,\beta} q^n,
\qquad
Z_\beta^{\mathrm{red}}(q) := \frac{Z_\beta^{\mathrm{DT}}(q)}{Z_0^{\mathrm{DT}}(q)}.
\]

Pandharipande--Thomas (PT) theory \cite{PT2009} provides an alternative curve-counting 
framework.  A stable pair consists of a pure one-dimensional sheaf 
$\mathcal{F}$ together with a section $s:\mathcal{O}_X\to\mathcal{F}$ such that 
$\mathrm{coker}(s)$ has zero-dimensional support.  The moduli space $P_n(X,\beta)$ of 
stable pairs with $[\mathcal{F}] = \beta$ and $\chi(\mathcal{F}) = n$ also carries a 
zero-dimensional virtual cycle, giving invariants
\[
P_{n,\beta} := \int_{[P_n(X,\beta)]^{\mathrm{vir}}} 1 \in \mathbb{Z},
\]
and a generating series
\[
Z_\beta^{\mathrm{PT}}(q) := \sum_{n\in\mathbb{Z}} P_{n,\beta} q^n.
\]

The DT/PT correspondence, first conjectured by Pandharipande and Thomas, was 
established independently by Bridgeland \cite{Bridgeland2011} and Toda 
\cite{Toda2010} using wall-crossing techniques in derived categories and Hall 
algebra methods:

\begin{theorem*}[DT/PT correspondence {\cite{Bridgeland2011, Toda2010}}]
\[
Z_\beta^{\mathrm{PT}}(q) = Z_\beta^{\mathrm{red}}(q).
\]
\end{theorem*}
\subsection*{Categorification of Donaldson--Thomas Theory}

Given the success of numerical curve-counting theories, a natural question arises: can these invariants be enhanced to capture more refined geometric or topological information? One promising direction is to develop motivic or cohomological refinements, which aim to lift integer-valued invariants to classes in the Grothendieck ring of varieties or to mixed Hodge structures, such that their Euler characteristics recover the classical numerical invariants.

Among the known curve-counting theories, DT theory appears particularly well-suited for such enhancements. The works of Kontsevich--Soibelman \cite{KS2008,KontsevichSoibelman0910} envisioned the possibility of defining motivic DT invariants, but their realization required additional data such as orientation and $d$-critical structures. The existence of global orientation data for moduli stack of objects on a Calabi--Yau threefold was unclear at the time, limiting the scope of these theories. 

Nevertheless, important progress towards motivic refinements of the DT/PT correspondence has been made in various local situations. For instance, Morrison--Mozgovoy--Nagao--Szendr\H{o}i compute the motivic Donaldson--Thomas theory of the resolved conifold and obtain a motivic DT/PT correspondence in that setting, while Davison--Ricolfi establish a local motivic DT/PT correspondence for curves in smooth threefolds via virtual motives of Quot schemes \cite{MorrisonMozgovoyNagaoSzendroi2012, DavisonRicolfiLocal}. These results work with explicit local models (such as quivers with potential and critical loci of Quot schemes), where vanishing cycles and orientation data can be handled directly, precisely because the existence of global orientation data had not yet been established.

Recent breakthroughs, however, have clarified many of these technical issues. The theory of shifted symplectic structures developed by \cite{PTVV2013} provides a derived geometric framework for moduli stacks. Building on this, Joyce and collaborators—including Bussi, Brav, Schürg, Meinhardt, Szendrői, Upmeier, Ben-Bassat, and Dupont—developed a rigorous theory of \(d\)-critical structures and orientation data on moduli stacks~\cite{Joyce2007, JoyceDcritical, BenBassatBravBussiJoyce2015, BravBussiDupontJoyceSzendroi2015, BravBussiJoyce2019, JoyceUpmeier2021}.

Building on these foundational developments, we may define the motivic Donaldson--Thomas and Pandharipande--Thomas invariants as elements in the ring of monodromic motives on algebraic stacks, using motivic Behrend function constructions and motivic integration:
\begin{definition*}[{\ref{Definitionofdtinvariants}}]
    We define
    \begin{align*}
        \textup{DT}^{\textup{mot}}_{n,\beta}(X) &= \int_{\mathcal{M}_{\textup{DT}}^{n,\beta}}(\L^{\frac{1}{2}}-\L^{-\frac{1}{2}})\nu_{\mathcal{M}_{\textup{DT}}^{n,\beta}}^\textup{mot}\in\hat{\mathbb{M}}^{\textup{mon}}(\mathbb{C}), \\
        \textup{PT}^{\textup{mot}}_{n,\beta}(X) &= \int_{\mathcal{M}_{\textup{PT}}^{n,\beta}}(\L^{\frac{1}{2}}-\L^{-\frac{1}{2}})\nu_{\mathcal{M}_{\textup{PT}}^{n,\beta}}^\textup{mot}\in\hat{\mathbb{M}}^{\textup{mon}}(\mathbb{C}).
    \end{align*}
\end{definition*}

This paper aims to refine the DT/PT correspondence to the motivic level.
We follow Toda's approach in \cite{Toda2010, toda2020hall}, which realizes the DT/PT correspondence as a wall-crossing phenomenon for weak stability conditions on a suitable triangulated category.  In Toda's proof of the numerical DT/PT correspondence, one first establishes an identity in the Hall algebra of stack functions relating the DT and PT moduli stacks,
and then applies an integration map.  Since this integration map is
compatible with the Hall algebra product, the desired DT/PT identity for
numerical invariants follows.

In this paper we upgrade this integration map to a motivic
integration map and show that it is again compatible with the Hall algebra
product.  The key new input is Bu's recent motivic integral identity
\cite{Bu2024AMI}, together with the existence of good moduli spaces and
their local structure.  As a consequence, we obtain a motivic refinement
of the DT/PT correspondence.  The following is the main result of this
paper, establishing an identity between motivic DT and PT invariants:

\begin{maintheorem}[= Theorem \ref{main}]
    Let $X$ be a Calabi--Yau threefold over $\mathbb{C}$. Then
    \[
        \sum_{n,\beta}\textup{DT}^{\textup{mot}}_{n,\beta}(X)\,q^n x^\beta
        \;=\;
        \left(\sum_{n\geq0}\textup{DT}^{\textup{mot}}_{n,0}(X)\,q^n\right)
        \left(\sum_{n,\beta}\textup{PT}^{\textup{mot}}_{n,\beta}(X)\,q^n x^\beta\right).
    \]
\end{maintheorem}

A second main ingredient of the paper is the existence of good moduli
spaces for objects that are semistable with respect to the weak stability
condition on the DT/PT wall.

\begin{maintheorem}[= Theorem \ref{ssgoodmodulispace}]
    The moduli stack $\M^{ss}$ of semistable objects with respect to the
    DT/PT wall admits a good moduli space
    \[
        \pi:\M^{ss}\longrightarrow M^{ss}.
    \]
\end{maintheorem}

Moreover, the motivic DT/PT identity can be refined to the level of the
good moduli space of semistable objects: we prove an identity of motivic
functions on $M^{ss}$.

\begin{maintheorem}[= Theorem \ref{mian2}]
    There is an equality of motivic functions over the good moduli space
    $M^{ss}$:
    \[
        \sum_{n,\beta}\mathcal{DT}^{\mathrm{mot}}_{n,\beta}
        \;=\;
        \left(\sum_{n\geq0}\mathcal{DT}^{\mathrm{mot}}_{n,0}\right)
        \cdot
        \left(\sum_{n,\beta}\mathcal{PT}^{\mathrm{mot}}_{n,\beta}\right),
    \]
    where $\mathcal{DT}_{n,\beta}^\text{mot}$, $\mathcal{PT}_{n,\beta}^\text{mot}$ are some motives over the good moduli space that refine the motivic DT/PT invariants.
\end{maintheorem}
Pushing this identity forward along $M^{ss}\to\Spec\mathbb{C}$ recovers the
motivic DT/PT correspondence of Theorem~\ref{main}.

\subsection*{Acknowledgements}
I am grateful to my advisor Hsueh-Yung Lin for his guidance in algebraic geometry and continued support of my research. I am especially grateful to Tasuki Kinjo, who suggested the topic of this thesis, for his patient guidance on Donaldson–Thomas theory and for kindly hosting my visit at RIMS, and to Yukinobu Toda, whose foundational work on the DT/PT correspondence underpins this thesis, for generously hosting my stay at IPMU and for answering many questions about his results. I also thank Adeel Khan and Wu-Yen Chuang for many helpful discussions. Finally, I acknowledge the support of the National Center for Theoretical Sciences (NCTS) for funding my visits to RIMS and IPMU.
\subsection*{Notations and conventions}
    \begin{itemize}
    \item We work over the field of complex numbers $\mathbb{C}$ throughout the paper.
    \item By an algebraic variety, we mean a finite type separated $\mathbb{C}$-scheme.
    \item A Calabi--Yau threefold $X$ is a smooth projective threefold with trivial canonical bundle, satisfying $H^1(X,\O_X)=0$.
    \item Throughout this paper we work with schemes, algebraic spaces, and algebraic stacks that are quasi-separated and locally of finite type; in addition, every algebraic stack is assumed to have a separated diagonal.
    \item We denote by $\mathbb{S}$ the $\infty$-category of spaces; see \cite{LurieHTT}.
    \item We follow Toen's definition of derived algebraic stacks. A derived algebraic stack is a derived stack that is $n$-geometric for some $n$ and locally of finite presentation (see \cite{ToenVezzosi2008}); in our convention, its cotangent complex is perfect, and the rank of that complex is called the virtual dimension.
\end{itemize}

\section{Stability conditions and moduli spaces}
In this section, we recall some basic facts about DT/PT moduli spaces. We fix a smooth projective Calabi--Yau threefold $X$, and consider its bounded derived category of coherent sheaves $D^b(X):=D^b(\textup{Coh}(X))$. As \cite{PT2009} suggests, the DT/PT correspondence may be realized as a wall crossing in the derived category $D^b(X)$. However, the existence of a Bridgeland stability condition on the entire derived category \(D^{\mathrm{b}}(X)\) is not known for a general Calabi--Yau threefold. Therefore, following Toda's approach \cite{Toda2010}, we restrict to a triangulated subcategory \(\mathcal{D}_X \subset D^{\mathrm{b}}(X)\) and work with a weak stability condition on it. In the first section we review weak stability conditions in the sense of \cite{Toda2010}.

\subsection{Weak stability conditions on triangulated categories}
There are generally two approaches to constructing weak stability conditions: one via slicings and the other via hearts of $t$-structures. In our case, we adopt only the heart-based approach and thus do not define slicings in this paper. For details, we refer the reader to \cite{Toda2010}.
Given a triangulated category $\mathcal{D}$ and its Grothendieck group $K(\mathcal{D})$ together with a finitely generated $\mathbb{Z}$-module $\Gamma$ and a $\mathbb{Z}$-linear homomorphism,
\[
\textup{cl}:K(\mathcal{D})\to\Gamma
\]
and also a filtration on $\Gamma$,
\[
\Gamma_\bullet:0\subsetneq\Gamma_0\subsetneq\cdots\subsetneq\Gamma_N=\Gamma,
\]
such that each quotient 
\[
\mathbb{H}_i:=\Gamma_i/\Gamma_{i-1}
\]
is a free abelian group. For each $i$, we set $\mathbb{H}_i^{\vee}:=\Hom_\mathbb{Z}(\mathbb{H}_i,\mathbb{C})$ and fix norms $\|*\|_i$ on $\mathbb{H}_i\otimes_\mathbb{Z}\mathbb{R}$. Let $Z=\{Z_i\}_{i=0}^N\in\prod_{i=0}^N\mathbb{H}_i^{\vee}$ be a set of functions. We define a map $Z:\Gamma\to\mathbb{C}$ as follows. For $v\in \Gamma$, we take the unique integer $i\in\{0,\dots,N\}$ such that $v\in\Gamma_i$, but $v\not\in\Gamma_{i-1}$, define $Z(v)=Z_i([v])$ where $[v]$ denote the class of $v$ in $\mathbb{H}_i$, and denote $\|E\|:=\|[\textup{cl}(E)]\|_i$.

The function $Z$ can be seen as a function on $K(\mathcal{D})\to\mathbb{C}$ by composing $\textup{cl}:K(\mathcal{D})\to\Gamma$. We will write $Z(E)=Z(\textup{cl}(E))$ for short.
\begin{definition}[\cite{Toda2010}]\label{stabilityconditiondef}
    Let $\textup{Slice}(\mathcal{D})$ be the set of slicing on $\mathcal{D}$ that are locally finite. The space of weak stability conditions $\textup{Stab}_{\Gamma_\bullet}(\mathcal{D})$ is defined to be the collection of pairs $(Z,\mathcal{P})\in\prod_{i=0}^{N}\mathbb{H}_i^{\vee}\times\textup{Slice}(\mathcal{D})$, subject to the following conditions.
    \begin{enumerate}
        \item[(i)] For any non-zero $E\in\mathcal{P}(\phi)$, we have
        \[
        Z(E)\in\mathbb{R}_{>0}\exp(i\pi\phi).
        \]
        \item[(ii)] There is a constant $C>0$ such that for any non-zero $E\in\bigcup_{\phi\in\mathbb{R}}\mathcal{P}(\phi)$, we have
        \[
        \|E\|\leq C|Z(E)|.
        \]
    \end{enumerate}
\end{definition}
\subsection{Constructions of weak stability conditions}
\begin{definition}
    Fix a bounded t-structure on $\mathcal{D}$ and let $\A$ be its heart. A weak stability function on $\A$ is a function $Z\in\prod_{i=0}^N\mathbb{H}_i^\vee$ such that for any non-zero $E\in\A$, we have
    \[
    Z(E)\in\{r\exp(i\pi\phi):r>0,0<\phi\leq1\}.
    \]
    In particular, for any $0\neq E\in\A$ there is a well-defined slope, $\arg Z(E)\in(0,\pi]$.
\end{definition}

\begin{definition}
    Given a weak stability function $Z\in\prod_{i=0}^N\mathbb{H}_i^{\vee}$ on $\A$. An object $0\neq E\in\mathcal{A}$ is $Z$-semistable (resp. stable) if it satisfies the following condition.
    \begin{itemize}[label=\textbullet]
        \item For any exact sequence
    \[
\begin{tikzcd}
	0 & F & E & G & 0
	\arrow[from=1-1, to=1-2]
	\arrow[from=1-2, to=1-3]
	\arrow[from=1-3, to=1-4]
	\arrow[from=1-4, to=1-5]
\end{tikzcd}
\]
in $\A$, we have
\[
\arg Z(F)\leq\arg Z(G)\quad (\textup{resp.} \arg Z(F) <\arg Z(G) ).
\]
    \end{itemize}
    
\end{definition}

\begin{definition}[Harder--Narasimhan property]
    Given a weak stability function $Z\in\prod_{i=0}^N\mathbb{H}_i^{\vee}$ on $\A$. We say that $Z$ has Harder--Narasimhan property if for any object $E\in\A$ there is a $\mathbb{Z}$-filtration
    \[
    0=E_0\subset E_1\subset\cdots\subset E_n=E
    \]
    such that each subquotient $F_j=E_j/E_{j-1}$ is $Z$-semistable with
    \[
    \arg Z(F_1)> \arg Z(F_2)>\cdots>\arg Z(F_n).
    \]
\end{definition}
It is not easy to construct a slicing on a triangulated category. Nevertheless, it is often more natural to consider the heart of a t-structure. The following proposition shows that one can construct a stability condition using a heart and a stability function defined on it.
\begin{proposition}[\cite{Toda2010}]
    Giving a heart of a bounded t-structure on $\mathcal{D}$ and a weak stability function on it with the H-N property is the same as giving a pair $(Z,\mathcal{P})\in\prod_{i=0}^N\mathbb{H}^{\vee}_i\times\textup{Slice}(\mathcal{D})$ satisfying the condition (i) in \ref{stabilityconditiondef}.
\end{proposition}
Denotes $\textup{Coh}_{\leq1}(X)$ the abelian subcategory of coherent sheaves with support dimension less than or equal to 1 on $X$. We consider the triangulated subcategory
\[
\mathcal{D}_X:=\langle\O_X,\textup{Coh}_{\leq1}(X)\rangle_{\textup{tr}}\subset D^b(\textup{Coh}(X)).
\]
This will be the main triangulated category we consider.

\subsection{Weak stability condition on $\mathcal{D}_X$}
We begin by fixing notation and definitions.  Denote by \(N_1(X)\) the abelian group of numerical curve classes on \(X\) and by \(A(X)\) its ample cone.  Set
\[
N_{\le 1}(X)=\mathbb{Z}\oplus N_1(X),
\qquad
\Gamma:= N_{\le 1}(X)\oplus\mathbb{Z}.
\]
Finally, let \(K(\mathcal{D}_X)\) be the Grothendieck group of the triangulated category \(\mathcal{D}_X\).

\begin{lemma}
    The homomorphism $\textup{cl}:K(\mathcal{D}_X)\to\Gamma$ given by 
    \[
    E\mapsto (\textup{ch}_3(E),\textup{ch}_2(E),\textup{ch}_0(E))
    \]
    is well-defined.
\end{lemma}
\begin{proof}
    We need to show that $\textup{ch}_*$ has integer coefficients for $*=0,2,3$ on $\mathcal{D}_X$. By definition of $\mathcal{D}_X$, its grothendieck group is generated by $\O_X$ and $\textup{Coh}_{\leq1}(X)$, we only need to show that $\textup{ch}_3(E)$, $\textup{ch}_2(E)$ has integer coefficient for $E\in\textup{Coh}_{\leq1}(X)$. By HRR, we have
    \[
    \textup{ch}_3(E)=\int_X(\textup{ch}_2(E)+\textup{ch}_3(E))\cdot\left(1+\frac{\textup{c}_2(X)}{12}\right)=\chi(E)\in\mathbb{Z}
    \]
    On the other hand, since $E$ has support dimension $\leq1$, we know that $\textup{c}_1(E)=0$.
    Hence 
    \[
    \textup{ch}_2(E)=\frac{\textup{c}_1(E)^2-2\textup{c}_2(E)}{2}=-\textup{c}_2(E).
    \]
\end{proof}
We now recall the construction of weak stability condition on $\mathcal{D}_X$.
Consider the following filtration on $\Gamma$
\[
\Gamma_\bullet:0\subsetneq\mathbb{Z}\subsetneq N_{\leq1}\subsetneq\Gamma
\]
where each inclusion is the natural inclusion of direct summand.
Then we have $\mathbb{H}_0=\mathbb{Z}$, $\mathbb{H}_1=N_1(X)$ and $\mathbb{H}_2=\mathbb{Z}$. Consider
\[
\xi=(-z_0,-i\omega,z_1)
\]
where $z_0,z_1\in\{r\exp(i\pi\theta) \, |\,r>0,\theta\in(\pi/2,\pi)\}$ and $\omega\in A(X)$. 

For $(s,l,r)\in\mathbb{Z}\oplus N_1(X)\oplus\mathbb{Z}$ we define $Z_\xi=\{Z_{\xi,i}\}_{i=0}^2\in\mathbb{H}_0^\vee\times\mathbb{H}_1^\vee\times\mathbb{H}_2^\vee$ by
\[
Z_{\xi,0}(s)=-z_0\cdot s, Z_{\xi,1}(l)=-i\omega\cdot l, Z_{\xi,2}(r)=z_1\cdot r.
\]
Thus, we have a family of functions $Z_\xi$.
\begin{proposition}[{\cite[Lemma 3.5]{Toda2010}}]
        Let
\[
\mathcal{A}_X := 
\bigl\langle \mathcal{O}_X,\;\mathrm{Coh}_{\le 1}(X)[-1] \bigr\rangle_{\mathrm{ex}}
\subset \mathcal{D}_X
\]
be the smallest extension-closed subcategory of \(\mathcal{D}_X\) containing
\(\mathcal{O}_X\) and \(\mathrm{Coh}_{\le 1}(X)[-1]\).
Then \(\mathcal{A}_X\) is the heart of a bounded t-structure on \(\mathcal{D}_X\).

\end{proposition}
\begin{lemma}[{\cite[Lemma 2.17 + 3.8]{Toda2010}}]
    The association $\xi\mapsto (Z_\xi,\mathcal{A}_X)$ determines a continuous family in $\textup{Stab}_{\Gamma_\bullet}(\mathcal{D}_X)$. We denote this continuous family $\mathcal{V}_X$.
\end{lemma}
\subsection{Moduli stacks of semistable objects}
We now collect some foundational properties of the moduli stack of objects in $\mathcal{A}_X$. As shown in \cite{Max2006}, there exists an algebraic stack $\mathcal{M}$, which parametrizes objects $E \in D^b(X)$ satisfying
\[
\mathrm{Ext}^i(E, E) = 0 \quad \textup{for all } i < 0.
\]
Consider the determinant map 
\[
\det:\M\to\textup{Pic}(X), E\mapsto\det(E).
\]
The objects in $\mathcal{D}_X$ are contained in the fiber at $[0]=\O_X\in\textup{Pic}(X)$. We denote $\M_0$ the fiber of $\det$ at $[0]$. Then the moduli stack of objects $\O bj(\mathcal{A}_X)$, as an abstract stack, is a substack of $\M_0$. The stack $\O bj(\A_X)$ admits a graded decomposition 
\[
\coprod_{v\in\Gamma}\O bj^v(\mathcal{A}_X)=\O bj(\mathcal{A}_X)
\]
where $\O bj^v(\mathcal{A}_X)$ is the moduli stack of objects $E\in\mathcal{A}_X$ with $\textup{cl}(E)=v\in\Gamma$. For a weak stability condition of the form $\sigma_\xi=(Z_\xi,\mathcal{A}_X)$ and a numerical type $v$ we may consider its moduli stack of semistable objects $\M^v(\sigma_\xi)\subset\O bj^v(\mathcal{A}_X)$.
We would like to focus on objects of rank one or has dimension zero. Define
\begin{align*}
    &\Gamma^1=\{v\in\Gamma\,|\, v=(-n,-\beta,1),\beta:\textup{effective}\}\\
    &\Gamma^0=\{v\in\Gamma\,|\, v=(-n,0,0),n\geq0\}
\end{align*}

\begin{theorem}[\cite{Toda2010}]
    Let $v\in\Gamma^0\cup\Gamma^1$ be a numerical type.
    \begin{enumerate}
        \item the moduli stack of objects
        \[
        \O bj^v(\mathcal{A}_X)\subset\M_0
        \]
        is open in $\M_0$. In particular, $\O bj^v(\A_X)$ is an algebraic stack locally of finite type over $\mathbb{C}$.
        \item For any weak stability condition of the form $\sigma_\xi$, the substack of $\sigma_\xi$-semistable objects
        \[
        \M^v(\sigma_\xi)\subset\O bj^v(\mathcal{A}_X)
        \]
        is open in $\O bj^v(\A_X)$ of finite type over $\mathbb{C}$.
    \end{enumerate}
\end{theorem}
\begin{theorem}[{\cite[Proposition 3.12]{Toda2010}}]\label{M=I/G}
    For a stability condition of the form $\sigma_\xi$ where $\xi=(-z_0,-i\omega,z_1)$. For $v=(-n,-\beta,1)\in\Gamma^1$ we have the following
    \begin{enumerate}
        \item Suppose that $\arg z_0<\arg z_1$. We have
        \[
        \M^v(\sigma_\xi)=[I_n(X,\beta)/\mathbb{G}_m]
        \]
        where $I_n(X,\beta)$ is the moduli space of ideal sheaves $I_C\subset\O_X$ with numerical type $(-n,-\beta,1)$.
        \item Suppose that $\arg z_0>\arg z_1$. We have
        \[
        \M^v(\sigma_\xi)=[P_n(X,\beta)/\mathbb{G}_m]
        \]
        where $P_n(X,\beta)$ is the moduli space of two term complexes $(\O_X\overset{s}{\to}\F)$ for stable pairs with numerical type $(-n,-\beta,1)$.
    \end{enumerate}
    In both situations, the $\mathbb{G}_m$-action is trivial.
\end{theorem}
We denote the above two moduli stacks as $\M^{n,\beta}_{\textup{DT}}$ and $\M^{n,\beta}_{\textup{PT}}$ when $v=(-n,-\beta,1)$. For a stability condition $\sigma_\xi$ on the wall (i.e. $\arg(z_0)=\arg(z_1)$), we denote $\M^{ss,v}=\M^v(\sigma_\xi)$ for $v\in\Gamma^1\cup\Gamma^0$ and 
\[
\M^{ss}=\coprod_{v\in\Gamma^1\cup\Gamma^0}\M^{ss,v}.
\]
\subsection{The Good moduli space and existence theorem}
In \cite{alper2013good}, Alper introduced the notion of a good moduli space as a fundamental concept in intrinsic GIT. Subsequently, the work of Alper--Halpern-Leistner--Heinloth \cite{AHLH18} formulated sufficient criteria for good moduli spaces to exist, while Alper--Hall--Rydh \cite{Alper2019TheL} established a structural theorem. For our purposes, the structural theorem guarantees the existence of a Nisnevich local structure, which allows motivic properties to be reduced to local situations. This local structure is an essential ingredient for the motivic integral identity \cite{Bu2024AMI}. 
\begin{definition}
    We say that an algebraic stack $\X$ admits a good moduli space if there is a morphism $\phi:\X\to Y$, where $Y$ is an algebraic space satisfied the following conditions
    \begin{enumerate}
        \item The pushforward $\phi_*:\textup{QCoh}(\X)\to\textup{QCoh}(Y)$ is an exact functor.
        \item The natural map $\O_Y\to\phi_*\O_{\X}$ is an isomorphism.
    \end{enumerate}
\end{definition}
We now introduce two important notions that will be used to state the existence criterion.
Let $\Theta:=[\mathbb{A}^1/\mathbb{G}_m]$. For any scheme $S$ we denote $\Theta_S$ the base change $\Theta\times S$. For a DVR $R$ with fraction field $K$ and a uniformizer $\pi$, let $0$ be the unique closed point of $\Theta_R$.
\begin{definition}
    An algebraic stacks $\X$ is said to be $\Theta$-reductive if for every DVR $R$, any commutative diagram of solid arrows
    \[\begin{tikzcd}
	\Theta_R\setminus\{0\} & \X \\
	\Theta_R
	\arrow[from=1-1, to=1-2]
	\arrow[from=1-1, to=2-1]
	\arrow[dashed, from=2-1, to=1-2]
\end{tikzcd}\]
can be completed in a unique way.
\end{definition}
\begin{remark}\label{thetareductive}
We cover \(\Theta\setminus\{0\}\) by the two open subsets
\[
\Theta_K
\quad\textup{and}\quad
\bigl[(\mathbb{A}^1\setminus\{0\})/\mathbb{G}_m\bigr]\times\textup{Spec} R
\;\cong\;\textup{Spec} R.
\]
Accordingly, specifying a morphism
\(\Theta_R\setminus\{0\}\to\mathcal{X}\)
is equivalent to giving morphisms
\(\Theta_K\to\mathcal{X}\) and \(\textup{Spec} R\to\mathcal{X}\) together with an identification of their restrictions over \(\textup{Spec} K\).

\end{remark}
Set
\[
\overline{\mathrm{ST}}_R
\;:=\;
\bigl[\textup{Spec}\bigl(R[s,t]/(st-\pi)\bigr)\big/\mathbb{G}_m\bigr],
\]
where \(s\) and \(t\) are given \(\mathbb{G}_m\)-weights \(1\) and \(-1\), respectively.  
The equations \(s = t = 0\) define a closed point, which we denote by \(0\).
\begin{definition}
An algebraic stack \(\mathcal{X}\) is called \textbf{S}-complete if, for every DVR \(R\) and every commutative diagram of solid arrows
\[
\begin{tikzcd}
\overline{\mathrm{ST}}_{R}\setminus\{0\} \arrow[r] \arrow[d] &
\mathcal{X} \\[6pt]
\overline{\mathrm{ST}}_{R} \arrow[ru, dashed] &
\end{tikzcd}
\]
there exists a unique dotted arrow completing the diagram.
\end{definition}
\begin{remark}\label{Scomplete}
    The open substack \(\overline{\mathrm{ST}}_{R}\setminus\{0\}\) is covered by the loci
\(\{s\neq0\}\) and \(\{t\neq0\}\).  Observe that
\[
\Spec\bigl(R[s,t]_{s}/(st-\pi)\bigr)\big/\mathbb{G}_m
\;\cong\;
\Spec\bigl(R[s,t]_{s}/(t-\tfrac{\pi}{s})\bigr)\big/\mathbb{G}_m
\;\cong\;
\Spec\bigl(R[s]_{s}\bigr)\big/\mathbb{G}_m
\;\cong\;
\Spec R,
\]
and similarly
\[
\Spec\bigl(R[s,t]_{t}/(st-\pi)\bigr)\big/\mathbb{G}_m
\;\cong\;
\Spec\bigl(R[t]_{t}\bigr)\big/\mathbb{G}_m
\;\cong\;
\Spec R.
\]
Hence specifying a morphism
\(\overline{\mathrm{ST}}_{R}\setminus\{0\}\to\mathcal{X}\)
is equivalent to giving two morphisms \(\textup{Spec} R\to\mathcal{X}\) together with an isomorphism of their restrictions over \(\textup{Spec} K\).

\end{remark}
We can now state the existence result.
\begin{theorem}[J. Alper, D. Halpern Leistner and J. Heinloth \cite{AHLH18}]\label{AHLHmaintheorem}
    An algebraic stack $\X$ of finite type with affine stabilizers admits a good moduli space if and only if $\X$ is $\Theta$-reductive and \textbf{S}-complete.
\end{theorem}

\subsection{The existence of Good moduli spaces for Semistable objects}
We fix a stability condition $\sigma_0\in\mathcal{V}_X$ on the wall i.e. $\arg(z_0)=\arg(z_1)=:\eta$. Note that $\arg\eta>1/2$. So an object $E$ with slope greater than $1/2$ must be either of positive rank or has $0$-dimensional support. In this subsection we prove the following theorem.
\begin{theorem}\label{ssgoodmodulispace}
        For any numerical type $v\in\Gamma^0\cup\Gamma^1$. The stack  $\M^{ss,v}$ admits a separated good moduli space.
\end{theorem}
We follow a similar strategy in \cite{AHLH18}, first show that the moduli stack of objects in $\A_X$ is \textbf{S}-complete and $\Theta$-reductive with respect to DVRs essentially of finite type, and then we verify that these properties descend to the semistable locus.

We recall the notion of tilting heart following \cite{toda2020hall}. Let $\omega$ be an ample class in $H^2(X,\mathbb{Z})$. We define a slope function $\mu=\mu_\omega:K_0(X)\to\mathbb{Q}\cup\{\infty\}$ by
\[
\mu(E)=\mu_\omega(E):=\frac{c_1(E) \cdot \omega^2}{\textup{rank}(E)}.
\]
Then we have subcategories
\begin{align*}
    \textup{Coh}_{\mu\leq0}(X)&:=\langle E\in\textup{Coh}(X):E \textup{ is }\mu \textup{ semistable and } \mu(E)\leq0  \rangle,\\
\textup{Coh}_{\mu>0}(X)&:=\langle E\in\textup{Coh}(X):E \textup{ is }\mu \textup{ semistable and } \mu(E)>0  \rangle.
\end{align*}
These two subcategories  defines a torsion pair 
\[
\textup{Coh}(X)=\langle \textup{Coh}_{\mu>0}(X), \textup{Coh}_{\mu\leq0}(X)\rangle
\]
and thus we may take its tilting
\[
\A=\langle \textup{Coh}_{\mu\leq0}(X),\textup{Coh}_{\mu>0}(X)[-1]\rangle.
\]

By the discussions in \cite[Section~3]{Piyaratne2015ModuliOB}, the heart $\mathrm{Coh}(X)$ together with the classical slope function is a very weak stability condition, satisfies a BG inequality and good in the sense of \cite{Piyaratne2015ModuliOB}. 
Then \cite[Proposition~4.11, 4.14]{Piyaratne2015ModuliOB} guarantee the tilting heart $\A$ is noetherian, bounded and satisfying generic flatness property.
\begin{proposition}\label{MAisStheta}
    The moduli stack $\M_\A$ is an open substack of the moduli stack of universally gluable perfect complexes $\M$ and is \textbf{S}-complete and $\Theta$-reductive with respect to DVR $R$ essentially of finite type.
\end{proposition}
\begin{proof}
    This follows directly from the the fact that $\A$ satisfies the generic flatness condition and \cite[Proposition 6.2.7, 6.2.10]{HalpernLeistner2014} (see also \cite[Proposition 7.16, 7.17]{AHLH18}).
\end{proof}

\begin{lemma}\label{identificationlemma}
    Let $\M_\A'$ be the closed and open substack of object in $\A$ of rank $1$ and $c_1=0$, then we have an identification of stacks
    \[
    \M_\A'=\coprod_{\textup{rk}(v)=1}\O bj^{v}(\A_X)
    \]
\end{lemma}
\begin{proof}
    We prove that there is an identification of subcategories
    \[
    \A_X^{\textup{rk}=1}=\A^{\textup{rk}=1,c_1=0}
    \]
    where left hand side is the category of rank $1$ objects in $\A_X$ and the right hand side is the category of rank $1$ objects with vanishing $c_1$.
    Since $\A$ is an abelian category, the inclusion $\A_X^{\textup{rk}=1}\subset\A^{\textup{rk}=1,c_1=0}$ is obvious. We show the converse. Let $E$ be an object in $\A^{\textup{rk}=1,c_1=0}$ there is a distinguished triangle, by the construction of torsion pair,
    \[
    F\to E\to T[-1]
    \]
    where $F\in\textup{Coh}_{\mu\leq0}(X)$, $T\in\textup{Coh}_{\mu>0}(X)$.

    Since $c_1(F)-c_1(T)=c_1(E)=0$, we have $c_1(F)=c_1(T)=0$. This implies that $\mu(E)=0$ and $\mu(T)=\infty$. A $\mu$-semistable sheaf of rank $1$ with vanishing $c_1$ is an ideal sheaf $I_C$ for some closed subscheme $C$ and a $\mu$-semistable sheaf of rank $0$ with vanishing $c_1$ is a coherent sheaf of support dimension $\leq1$. Thus $E$ is an extension of objects in $\A_X$ and hence lies in $\A_X$.
\end{proof}

The following proposition provides a more practical criterion.
\begin{proposition}\label{criterionab}
    A quasi-compact open substack $\mathcal{U}\subset\M_\A$ admits a separated good moduli space if the following statements holds.
    \begin{enumerate}
        \item[(a)]($\Theta$-reductive) Let $R$ be a DVR essentially of finite type with fraction field $K$ and residue field $\kappa$. For every $R$-point $E$ of $\M_\A$ (i.e., an object in $D^b(X_R)$ perfect relative in $R$ and $R$-flat) and every $\mathbb{Z}$-graded filtration 
        \[
        0\subset \cdots\subset E_{i-1}\subset E_i\subset E_{i+1}\subset\cdots\subset E
        \]
        satisfying $E_i=0$ for $i<<0$ and $E_i=E$ for $i>>0$, such that $E_i/E_{i+1}$ are flat over $R$. If $E$ and $\textup{gr}(E_{\bullet}|_{K})$ are both in $\mathcal{U}$. Then $\textup{gr}(E_\bullet|_\kappa)$ is also in $\mathcal{U}$. 
        \item[(b)]($\mathbf{S}$-completeness) For any two $\kappa$-points $E$, $F$ of $\mathcal{U}$, if there are $\mathbb{Z}$-graded filtrations in $\M_A(\kappa)$
        \begin{align*}
              0\subset \cdots\subset E_{i-1}\subset E_i\subset E_{i+1}\subset\cdots\subset E\\
              F\supset \cdots\supset F^{i-1}\supset F^i\supset F^{i+1}\supset\cdots\supset 0
        \end{align*}
        satisfying $E_i=0$, $F^i=F$ for $i\ll0$ and $E_i=E$, $F^i=0$ for $i\gg0$ and $E_i/E_{i-1}\cong F^i/F^{i+1}$ for all $i$. Then $\textup{gr}(E_\bullet)$ is also in $\mathcal{U}(\kappa)$.
    \end{enumerate}
\end{proposition}
\begin{proof}
    By Theorem~\ref{AHLHmaintheorem}, it suffices to verify $\Theta$-reductivity and $\mathbf{S}$-completeness for $\mathcal{U}$. By Proposition~\ref{MAisStheta}, $\M_A$ is \textbf{S}-complete and $\Theta$-reductive, the $\Theta$-reductivity and \textbf{S}-completeness can be reduced to the existence of the following lifting diagrams
    \[\begin{tikzcd}
	\Theta_R\setminus\{0\} & \mathcal{U} && \overline{\mathrm{ST}}_R\setminus\{0\} & \mathcal{U} \\
	\Theta_R & \M_A && \overline{\mathrm{ST}}_R & \M_A
	\arrow[from=1-1, to=1-2]
	\arrow[from=1-1, to=2-1]
	\arrow[from=1-2, to=2-2]
	\arrow[from=1-4, to=1-5]
	\arrow[from=1-4, to=2-4]
	\arrow[from=1-5, to=2-5]
	\arrow[dashed, from=2-1, to=1-2]
	\arrow[from=2-1, to=2-2]
	\arrow[dashed, from=2-4, to=1-5]
	\arrow[from=2-4, to=2-5]
\end{tikzcd}\]
In other words, the lifting already exists in $\M_A$, we just need to make sure it lies in $\mathcal{U}$. And we only need to show the image of $0$ lies in $\mathcal{U}$.

We start from $\Theta$-reductivity. By \cite[Corollary 7.13.]{AHLH18}, a $\Theta_R$-point is the same as a $\mathbb{Z}$-gradred filtration
        \[
        0\subset \cdots\subset E_{i-1}\subset E_i\subset E_{i+1}\subset\cdots\subset E
        \]
        satisfying $E_i=0$ for $i<<0$ and $E_i=E$ for $i>>0$, such that $E_i/E_{i+1}$ are flat over $R$.
Since its restriction to $\Theta_R\setminus\{0\}$ maps to $\mathcal{U}$, we also know that $E$ and $\text{gr}(E_\bullet|_K)$ both lies in $\mathcal{U}$ and the restriction to $0$ corresponds to a $\left[\text{Spec}(\kappa)/\mathbb{G}_m\right]$-point.

We then consider \textbf{S}-completeness. The restriction to $(\pi=0)$ i.e. $\left[\text{Spec}(\kappa[s,t]/st)/\mathbb{G}_{m,\kappa}\right]$ corresponds to the above description i.e. opposite filtrations and the point corresponds to $0$ is $\textup{gr}(E_\bullet)$.
\end{proof}
This proposition originally appears in Alper's lecture notes \cite{AlperModuliNotes}, where it is used in the proof that the moduli stack of semistable vector bundles admits a good moduli space. In the original setting, the abelian category under consideration is that of coherent sheaves. Alper applies this proposition in combination with certain properties of semistable vector bundles to establish the existence of a good moduli space. Motivated by this strategy, we aim to first prove that certain semistable objects in our setting exhibit analogous properties.

\begin{lemma}\label{subobjectosss}
    Let $E$ be a $\sigma_0$-semistable object of rank $1$. Then any subobject of $E$  with slope $>1/2$ is also $\sigma_0$-semistable.
\end{lemma}
\begin{proof}
    Let $F$ be a subobject of $E$. If $F$ has rank $0$ then $F$ must be a shift of $0$-dimensional sheaf which is obviously semistable. So we may assume $F$ has rank $1$. If $F$ is not semistable, then there is an exact sequence in $\A_X$ that destabilize $F$
    \[
    \begin{tikzcd}
	0 & G & F & F/G & 0
	\arrow[from=1-1, to=1-2]
	\arrow[from=1-2, to=1-3]
	\arrow[from=1-3, to=1-4]
	\arrow[from=1-4, to=1-5]
    \end{tikzcd}
    \]
    where $F/G$ has slope $=1/2$. Now consider the cokernel of $G\to F\to E$ we get an exact sequence
    \[
    \begin{tikzcd}
	0 & F/G & E/G &E/F & 0
	\arrow[from=1-1, to=1-2]
	\arrow[from=1-2, to=1-3]
	\arrow[from=1-3, to=1-4]
	\arrow[from=1-4, to=1-5]
    \end{tikzcd}.
    \]
    Since $F/G$ has slope $=1/2$, the support of $F/G$ has dimension $>0$. So $G$ must has rank $1$. Thus $E/G$ is a rank $0$ object which contain $F/G$, so the support of $F/G$ also has dimension $>0$. Therefore, $E/G$ has slope $1/2$ which is a contradiction since $E$ is semistable.
\end{proof}
\begin{lemma}\label{quotientofss}
    Let $E$ be a $\sigma_0$-semistable object of rank $1$ of slope $>1/2$. Then any quotient of $E$ is also $\sigma_0$-semistable of slope $>1/2$. In particular, if the quotient has rank $0$ then it has $0$-dimensional support.
\end{lemma}
\begin{proof}
    Let $F$ be a subobject of $E$. If $F$ has rank $1$, then it has slope $>1/2$. By the semistability of $E$, the quotient $E/F$ must has slope $>1/2$. Hence support of $E/F$ is $0$-dimensional, which is $\sigma_0$-semistable. So we may assume $F$ has rank $0$. Now if there is an exact sequence destabilize $E/F$,
    \[
    \begin{tikzcd}
	0 & G & E/F & H & 0
	\arrow[from=1-1, to=1-2]
	\arrow[from=1-2, to=1-3]
	\arrow[from=1-3, to=1-4]
	\arrow[from=1-4, to=1-5]
    \end{tikzcd}.
    \]
    Then there is a surjective morphism $E\to E/F\to H$, where $H$ has slope $1/2$. Taking kernel of the morphism $E\to H$ then this $H$ also destabilize $E$ which is a contradiction.
\end{proof}
\begin{lemma}\label{directsumofwithpoint}
    The category of $\sigma_0$-semistable objects with slope $>1/2$ is closed under extension.
\end{lemma}
\begin{proof}
    Let $F$ be an extension of $\sigma_0$-semistable objects $E$ and $G$ in $\A_X$ satisfying
    \[
    \begin{tikzcd}
	0 & E & F & G & 0
	\arrow[from=1-1, to=1-2]
	\arrow[from=1-2, to=1-3]
	\arrow[from=1-3, to=1-4]
	\arrow[from=1-4, to=1-5]
    \end{tikzcd}.
    \]
    If there is an exact sequence destabilizes $F$, say
    \[
    \begin{tikzcd}
	0 & X & F & Y & 0
	\arrow[from=1-1, to=1-2]
	\arrow[from=1-2, to=1-3]
	\arrow[from=1-3, to=1-4]
	\arrow[from=1-4, to=1-5]
    \end{tikzcd}.
    \]
    where $X$ has slope $>1/2$ and $Y$ has slope $=1/2$. We may assume that $Y$ is semistable of slope $=1/2$. Then the composition $E\to F\to Y$ is zero. So $E\to F$ factor through $X$. Hence we obtain an exact sequence
    \[
    \begin{tikzcd}
	0 & X/E & G & Y & 0
	\arrow[from=1-1, to=1-2]
	\arrow[from=1-2, to=1-3]
	\arrow[from=1-3, to=1-4]
	\arrow[from=1-4, to=1-5]
    \end{tikzcd},
    \]
    which violates the semistability of $G$.
\end{proof}
Using the above lemmas, we now prove Theorem~\ref{ssgoodmodulispace}.
\begin{proof}[Proof of Theorem \ref{ssgoodmodulispace}]
    The result for $v\in\Gamma^0$ is well-known. We apply Proposition \ref{criterionab} on $\M^{ss,v}\subset\M_\A$ for $v\in\Gamma^1$. It suffices to show $(a)$ and $(b)$.
    \begin{enumerate}
        \item[(a)]($\Theta$-reductive) For any DVR $R$ with fraction field $K$ and residue field $\kappa$, for every $R$-point $E$ of $\M_\A$ and every $\mathbb{Z}$-graded filtration 
        \[
        0\subset \cdots\subset E_{i-1}\subset E_i\subset E_{i+1}\subset\cdots\subset E
        \]
        satisfying $E_i=0$ for $i\ll 0$ and $E_i=E$ for $i\gg0$, such that $E_i/E_{i+1}$ are $R$-point of $\M_\A$. If $E$ and $\textup{gr}(E_{\bullet}|_{K})$ are both in $\M^{ss,v}$. Then $\textup{gr}(E_\bullet|_\kappa)$ is also in $\M^{ss,v}$. 
        \begin{proof}
            First note that $(E_i/E_{i-1})|_K$ cannot have slope $1/2$ for any $i$, otherwise the projection $\textup{gr}(E_\bullet|_K)\to(E_i/E_{i-1})|_K$ would destabilize $\textup{gr}(E_\bullet|_K)$. By the openness of $\O bj^v(\A_X)$ for each $v$, the $(E_i/E_{i-1})|_\kappa$ cannot have slope $1/2$ for any $i$. Let $i$ be the unique index such that $(E_i/E_{i-1})|_\kappa$ has positive rank. Then $E_{i}|_\kappa$ has rank $1$ and hence has slope $>1/2$. By Lemma \ref{subobjectosss}, $E_i|_\kappa$ is semistable. Applying Lemma \ref{quotientofss}, we deduces that $(E_i/E_{i-1})|_\kappa$ is semistable. Finally, Lemma \ref{directsumofwithpoint} implies that $\textup{gr}(E_\bullet|_\kappa)$ is semistable.
        \end{proof}
        \item[(b)]($\mathbf{S}$-completeness) For any two $\kappa$-points $E$, $F$ of $\mathcal{M}^{ss,v}$, if there are $\mathbb{Z}$-graded filtrations in $\M_\A(\kappa)$
        \begin{align*}
              0\subset \cdots\subset E_{i-1}\subset E_i\subset E_{i+1}\subset\cdots\subset E\\
              F\supset \cdots\supset F^{i-1}\supset F^i\supset F^{i+1}\supset\cdots\supset 0
        \end{align*}
        satisfying $E_i=0$, $F^i=F$ for $i\ll0$ and $E_i=E$, $F^i=0$ for $i\gg0$ and $E_i/E_{i-1}\cong F^i/F^{i+1}$ for all $i$. Then $\textup{gr}(E_\bullet)$ is also in $\M^{ss,v}$.
        \begin{proof}
            Since $E$ is an object of rank $1$. There is a unique $i$ such that $E_i/E_{i-1}$ has rank $1$. Thus by Lemma \ref{directsumofwithpoint}, it suffices to show $E_j/E_{j-1}$ has $0$-dimensional support for all $j\neq i$ and $E_i/E_{i-1}$ is semistable. Using Lemma \ref{subobjectosss}, we see that $E_j$ is semistable of rank $1$ for all $j\geq i$ and $F_j$ is semistable of rank $1$ for all $j\leq i$. Now applying Lemma \ref{quotientofss} we get $E_j/E_{j-1}$ has $0$-dimensional support for all $j\neq i$ and $E_i/E_{i-1}$ is a rank $1$ semistable object.
        \end{proof}
    \end{enumerate}
    The existence of a good moduli space for $\M^{ss,v}$ then follows from Proposition \ref{criterionab}.
\end{proof}

\begin{corollary}\label{maincoro}
   The stack $\M^{ss}$ is Nisnevich locally fundamental, i.e., there exists a Nisnevich cover by stacks of the form $[U/\mathrm{GL}_n]$, where $U$ is an affine variety acted by $\mathrm{GL}_n$.
\end{corollary}
\begin{proof}
    This follows directly from the local structure theorem for algebraic stacks admitting a good moduli space \cite[Theorem 6.1]{Alper2019TheL}.
\end{proof}


\section{Shifted symplectic structure and orientation}
Compared to numerical Donaldson--Thomas invariants, defining motivic or cohomological Donaldson--Thomas invariants requires more information. Therefore, we need to introduce shifted symplectic structures and orientations. 
\subsection{Shifted symplectic structure}
We begin by briefly recalling some definitions and key existence results from the theory of shifted symplectic structures, as introduced in \cite{PTVV2013} via derived algebraic geometry; see also \cite{ParkIntroShiftedSymp} for a nice introduction.
\begin{definition}
    Given a derived stack $\dX$. We denote the space of $p$-forms, closed $p$-forms by
    \[
    \A^p(\dX,n),    \A^{p,cl}(\dX,n)\in\S
    \]
\end{definition}
\begin{remark}
    In contrast to the classical setting, where closed forms are defined as a subset of all forms (namely, those annihilated by the de Rham differential), in the derived framework a closed $p$-form of degree $n$ is not merely a form satisfying a condition. Rather, it consists of a $p$-form together with a specified "closing structure", a homotopy-theoretic datum encoding the vanishing of the differential in a coherent way. There is a forgetting morphism  $\A^{p,\textup{cl}}(\dX,n)\to\A^{p}(\dX,n)$ that forget the closing structure.
\end{remark}
The space of closed $p$-forms is not easy to describe. However, for $p$-forms we have the following proposition.
\begin{proposition}\label{form=map}
    Given a derived algebraic stack $\dX$, and its cotangent complex $\L_\dX$. We have an equivalence of spaces
    \[
    \A^p(\dX,n)\cong\textup{Map}(\mathcal{O}_{\dX},\wedge^p\L_{\dX}[n])\in\S.
    \]
\end{proposition}
Given a morphism of derived stacks $f:\dX\to\dY$ there is a pullback map of (closed) $n$-shifted $p$-form
\[
f^\star:\A^p(\dY,n)\to\A^p(\dX,n), f^\star:\A^{p,cl}(\dY,n)\to\A^{p,cl}(\dX,n)
\]
By Proposition~\ref{form=map}, any $2$-form $\omega$ of degree $n$ on a derived algebraic stack $\dX$ corresponds to a morphism of quasi-coherent complex
\[
\O_\dX\to\wedge^2\L_{\dX}[n]
\]
which gives rise to a morphism from tangent complex to $n$-shifted shifted cotangent complex
\[
\Theta_\omega:\mathbb{T}_\dX\to\L_{\dX}[n].
\]
\begin{definition}
    Let $\dX$ be a derived algebraic stack. An $n$-shifted symplectic structure on $\dX$ is a closed $2$-form $\omega\in\A^{p,cl}(\dX,n)$ of degree $n$ such that the underlying $2$-form induces an equivalence
    \[
    \Theta_\omega:\mathbb{T}_\dX\overset{\sim}{\to}\L_\dX[n]
    \]
\end{definition}
A derived algebraic stack $\dX$ with an $n$-shifted symplectic structure $\omega_\dX$ is called an $n$-shifted symplectic stack and we denote it by $(\dX,\omega_\dX)$. Given a pair of $n$-shifted symplectic stacks $(\dX,\omega_\dX), (\dY,\omega_\dY)$, there is an induced $n$-shifted symplectic structure on $\dX\times\dY$ given by 
\[
\omega_{\dX}\boxplus\omega_\dY=\textup{pr}_1^\star\omega_\dX+\textup{pr}_2^\star\omega_\dY
\]
where $\textup{pr}_1:\dX\times\dY\to\dX$, $\textup{pr}_2:\dX\times\dY\to\dY$ are the projections.
We now state the main existence theorem for shifted symplectic structures.
\begin{theorem}[{\cite[Theorem 2.5]{PTVV2013}}]\label{-1shiftedstronperf}
    Let $X$ be a smooth projective Calabi--Yau $n$-fold and let $\dY$ be a $d$-shifted symplectic derived Artin stack. Then the derived mapping stack
    \[
    \underline{\textup{dMap}}(X,\dY):(\textup{dSch}/\mathbb{C})\to\S, T\mapsto\textup{Map}(X\times T,\dY)
    \]
    carries a $(d-n)$-shifted symplectic structure from $\Y$. 
    
    In particular, the derived moduli stack $\textup{Perf}(X)$ of perfect complexes on a Calabi--Yau threefold $X$ carries a $(-1)$-shifted symplectic structure.
\end{theorem}

\subsection{d-critical structures}
The notion of a $d$-critical structure on an algebraic stack was introduced in \cite{BenBassatBravBussiJoyce2015}. They defined a canonical \'{e}tale sheaf $\mathcal{S}_X$ for each algebraic space $X$ and showed that $\mathcal{S}_X$ can be decompose into $\mathcal{S}_X=\mathcal{S}_{X}^0\oplus\mathbb{C}$. For a morphism $f:X\to Y$ there is a pullback map $f^\star:f^{-1}\mathcal{S}_{Y}^0\to\mathcal{S}_X^0$.

\begin{definition}[{\cite[Definition 2.5]{JoyceDcritical}}]
    Let $X$ be an algebraic space and $s\in\Gamma(X,\mathcal{S}_X^0)$ a section. 
    \begin{enumerate}
        \item An \'etale d-critical chart for $(X,s)$ is an \'etale morphism $\eta:R\to X$, a smooth scheme $U$, a closed embedding $i:R\hookrightarrow U$ with ideal sheaf $I$ of $i^{-1}\mathcal{O}_U$ defining $R$ and a morphism $f:U\to\mathbb{A}^1$ on $U$ with $f|_{i(R)^{red}}=0$, $i(R)=\textup{Crit}(f)$ and $f+I^2=s|_R$. We denote a d-critical chart as above by $(R,\eta,U,f,i)$.
        \item The section $s$ is called a d-critical structure on $X$ if for each point $x\in X$ there exists a d-critical chart $(R,\eta,U,f,i)$ such that the image of $\eta$ contains $x$.
    \end{enumerate}
\end{definition}
Given a smooth morphism $f:X\to Y$ of algebraic space and a d-critical structure $s\in\Gamma(Y,\mathcal{S}_Y^0)$. There is a pullback d-critical structure $f^\star s$ on $X$ \cite[Proposition 2.8]{JoyceDcritical}. Using this notion we may extend the definition of d-critical structure on algebraic stacks.

\begin{definition}
    Given an algebraic stack $\X$. The sheaf $\mathcal{S}_\X^0$ is defined to be the assignment 
    \[
    T\mapsto \Gamma(T,\mathcal{S}_T^0)
    \]
    in the lisse-\'etale topos on $\X$. A d-critical structure on $\X$ is a section $s\in\Gamma(\X,\mathcal{S}_\X^0)$ such that $s|_T\in\Gamma(T,\mathcal{S}_T^0)$ is a d-critical structure for each smooth morphism $T\to\X$.
\end{definition}
For an algebraic stack $\mathcal{X}$ equipped with a d-critical structure $s \in \Gamma(\mathcal{X}, \mathcal{S}_\mathcal{X}^0)$, we refer to the pair $(\mathcal{X}, s)$ as a d-critical stack. A morphism of d-critical stacks $f:(X,s)\to(Y,t)$ is a morphism $f$ such that $f^\star t=s$.

Let $\dX$ be a $(-1)$-shifted symplectic stack with classical truncation $\X=\dX_{\textup{cl}}$ an algebraic stack. By \cite[Theorem 3.18]{BenBassatBravBussiJoyce2015}, there exists an induced d-critical structure on the classical truncation $\X$. More precisely, there is a truncation functor
\[
F:
\left\{
\begin{array}{l}
\infty\textup{-category of } (-1)\textup{-shifted} \\
\textup{symplectic stacks } (\dX, \omega_\dX)
\end{array}
\right\}
\longrightarrow
\left\{
\begin{array}{l}
\textup{2-category of } \\
\textup{d-critical stacks } (\X, s)
\end{array}
\right\}.
\]
\subsection{Orientation}
\begin{definition}
    Given a $(-1)$-shifted symplectic stack $(\dX,\omega_\dX)$ over $\mathbb{C}$. An orientation of $\dX$ is a line bundle $K_\dX^{\frac{1}{2}}\to\dX$, together with an isomorphism $o_\dX:(K_{\dX}^{\frac{1}{2}})^{\otimes2}\cong K_\dX$, where $K_\dX$ is defined to be the determinant line bundle of the cotangent complex of $\dX$, which is called the canonical bundle of $\dX$. We called an $(-1)$-shifted symplectic stack with an orientation an oriented $(-1)$-shifted symplectic stack and denote it by $(\dX,\omega_\dX,o_\dX)$.
\end{definition}
Given a pair of oriented $(-1)$-shifted symplectic stacks $(\dX,\omega_\dX,o_\dX), (\dY,\omega_\dY,o_\dY)$, there is an induced orientation $\dX\times\dY$ given by the isomorphism
\[
o_{\dX\times\dY}:(K_{\dX}^{\frac{1}{2}}\boxtimes K_{\dY}^{\frac{1}{2}})^{\otimes2}\cong K_\dX\boxtimes K_\dY\cong K_{\dX\times\dY}.
\]
\begin{remark}
In \cite{BenBassatBravBussiJoyce2015}, the notion of the canonical line bundle \( K_\mathcal{X} \) for a \( d \)-critical stack \( (\mathcal{X}, s) \) is also defined. It is further shown that if \( (\mathcal{X}, s) \) arises as the classical truncation of a \( (-1) \)-shifted symplectic stack \( \dX \), there is a canonical isomorphism
\[
K_{\dX}|_{\mathcal{X}^{\mathrm{red}}} \cong K_\mathcal{X}|_{\mathcal{X}^{\mathrm{red}}}
\]
on the reduced underlying stack. For simplicity, we write \( K_\mathcal{X} \) to denote this restriction, even though it is technically a line bundle on \( \mathcal{X}^{\mathrm{red}} \).
\end{remark}
\begin{definition}
    The orientation on a d-critical stacks $(\X,s)$ is defined to be a line bundle $K_\X^{\frac{1}{2}}$ on $\X^{\textup{red}}$ together with an isomorphism $o:(K_{\X}^{\frac{1}{2}})^{\otimes2}\cong K_\X$. An oriented d-critical stack is a d-critical stack $(\X,s)$ together with an orientation $o$, we denote it by $(\X,s,o)$.
\end{definition}
\begin{lemma}[{\cite[Lemma 2.58]{JoyceDcritical}}]\label{inducedorientation}
    Given a smooth morphism $g:\Y\to\X$ of d-critical stacks. If $g$ is compatible with the d-critical structures. Then an orientation $K_\X^{\frac{1}{2}}$ of $\X$ induces an orientation of $\Y$ given by 
    \[
    K_\Y^{\frac{1}{2}}=g^*(K^{\frac{1}{2}}_\X)\otimes\textup{det}(\mathbb{L}_{\Y/\X})|_{\Y^{\textup{red}}}, g^\star o:(K_\Y^{\frac{1}{2}})^{\otimes 2}\cong K_\Y
    \]
\end{lemma}

\subsection{Results on the moduli stack of objects on a Calabi--Yau threefold}
In this section we focus on the moduli stack of perfect complex on a Calabi--Yau threefold.  Let $\dX$ be the derived moduli stack of perfect complexes on a Calabi--Yau 3-fold equipped with a $(-1)$-shifted symplectic structure $\omega_\dX$ given by Corollary \ref{-1shiftedstronperf}. Let $\Phi_2:\dX\times\dX\to\dX$ be the morphism corresponded to the direct sum of perfect complexes on $X$.

The existence of an orientation on the moduli stack of perfect complexes on a Calabi--Yau threefold is ensured by the following theorem. 
\begin{theorem}[{\cite[Theorem 3.6]{JoyceUpmeier2021}}]\label{orientationdata}
    There is an orientation $o:(K_\dX^{\frac{1}{2}})^{\otimes 2}\cong K_\dX$ on $\dX$ such that it is compatible with direct sums in the following sense: Let $\Phi_2:\dX\times\dX\to\dX$ be the morphism corresponded to the direct sum of perfect complexes on $X$.
    Then there is an isomorphism
    \begin{align*}
        &\lambda:\Phi_2^\star o\cong o\boxtimes o.
    \end{align*}
\end{theorem}
\begin{theorem}[{\cite[Corollary 8.19]{KinjoParkSafronov2024}}]\label{sss}
    The morphism $\Phi_2:\dX\times\dX\to\dX$ is compatible with $(-1)$-shifted symplectic structures, i.e. we have an equivalence
    \[
    \Phi_2^\star\omega_\dX\cong\omega_\dX\boxplus\omega_\dX
    \]
\end{theorem}
The following proposition gives open substack of $\dX_{\textup{cl}}$ a natural derived structure.
\begin{proposition}[{\cite[Proposition 2.1]{SchurgToenVezzosi2015}}]
    Let \(\dX\) be a derived stack and let \(\X := \dX_{\mathrm{cl}}\) denote its classical truncation.  
The truncation functor induces a one-to-one correspondence
\[
\phi_{\mathrm{cl}}\;:\;
\bigl\{\textup{Zariski-open derived substacks of }\dX\bigr\}
\;\xrightarrow{\;\sim\;}\;
\bigl\{\textup{Zariski-open substacks of }\X\bigr\}.
\]
In particular, every Zariski-open substack of \(\X\) carries a unique derived enhancement coming from \(\dX\).

\end{proposition}
By this proposition, all of our moduli stacks carries a natural derived structure from $\dX$. Thus an oriented $(-1)$-shifted symplectic structure on $\dX$ induced an oriented $(-1)$-shifted symplectic structure on $\M_\A$ that is compatible with direct sum.
\begin{corollary}
    There is an orientation on $\M_\A$ induced from $\mathfrak{X}$ that is compatible with direct sum and a d-critical structure that is compatible with the direct sum morphism $\M_\A\times\M_\A\to\M_\A$.
\end{corollary}
The above corollary implies that there is also an induced orientation and a d-critical structure on the open substack $\M^{ss}\subset\M_\A$.


\section{Motivic invariants and Integral identity}
In this section we recall motivic invariants for algebraic stacks, motivic Behrend functions and Bu's motivic integral identity for (-1)-shifted symplectic stacks \cite{Bu2024AMI}.

\subsection{Rings of motives}
We begin by reviewing the fundamental notion of motive rings over an algebraic stack, following \cite{Bu2024AMI,IntrinsicDTII}; for the original reference, see \cite{Joyce2007}.
We first define the Grothendieck ring of varieties over an algebraic stack $\X$. Recall that, by convention, the stack $\X$ is locally of finite type with affine stabilizers.
\begin{definition}
    The Grothendieck ring of varieties over $\X$ is defined to be the abelian group
    \[
    K_{\textup{var}}(\X)=\hat{\bigoplus_{Z\to\X}}\mathbb{Q}\cdot[Z\to\X]/\sim,
    \]
    where we sum over all isomorphism classes of morphisms $Z\to\X$, where $Z$ are varieties and $\hat{\bigoplus}$ means we allow infinite sum $\sum_Z n_Z[Z\to\X]$ but for any quasi-compact open substack $\mathcal{U}\subset\X$, there are only finitely many $Z$ such that $n_Z\neq0$ and $Z\times_\X\mathcal{U}\neq\varnothing.$ The relation is generated by $[Z]\sim [Z']+[Z\setminus Z']$ for closed subschemes $Z'\subset Z$.
\end{definition}
The Grothendieck ring of varieties carries a ring structure by taking fibre product over $\X$ on generators. It is also a commutative $K_\textup{var}(\mathbb{C})$-algebra, where $K_\textup{var}(\mathbb{C}):=K_\textup{var}(\textup{Spec}(\mathbb{C}))$, with the action induced by the homomorphism 
\[
K_\textup{var}(\mathbb{C})\to K_\textup{var}(\X),[Z]\mapsto[Z\times\X\to\X].
\] 
We denote $\L=[\mathbb{A}^1]\in K_\textup{var}(\mathbb{C})$.
\begin{definition}
    We consider the following localization
    \begin{align*}
    \mathbb{M}(\X)&=K_\textup{var}(\X)\hat{\otimes}_{K_\textup{var}(\mathbb{C})}\left(K_\textup{var}(\mathbb{C})[\L^{-1}]/(\L-1)\textup{-torsion}\right)\\
        \hat{\mathbb{M}}(\X)&=K_\textup{var}(\X)\hat{\otimes}_{K_\textup{var}(\mathbb{C})}K_\textup{var}(\mathbb{C})[\L^{-1},(\L^k-1)^{-1}],
    \end{align*}
where we invert $\L^k-1$ for all $k\geq1$ and $\hat{\otimes}$ means we allow infinite sum $\sum_{Z\to\X}[Z\to\X]\otimes n_Z$, but for any quasi-compact open substack $\mathcal{U}\subset\X$, there are only finitely many $Z$ such a that $n_Z\neq0$ and $Z\times_\X\mathcal{U}\neq\varnothing.$
\end{definition}
For a finite type morphism $\Y\to\X$ of algebraic stack, where $\Y$ has affine stabilizers  one can associate a class $[\Y\to\X]\in\hat{\mathbb{M}}(\X)$, which agrees with the usual one when $\Y$ is just a variety. The construction is as follows: for such $\Y$ there is a stratification of $\Y$ by locally closed substack $\Y=\coprod_i\Y_i$ such that $\Y_i=[U_i/\textup{GL}(n_i)]$, where $U_i$ is a quasi affine scheme with a $\textup{GL}(n_i)$ action, thus we define
\[
[\Y]:=\sum_{i}\frac{1}{[\textup{GL}(n_i)]}\cdot[U_i]\in\hat{\mathbb{M}}(\X),
\]
where $[\textup{GL}(n_i)]=\prod_{k=0}^{n_i-1}(\L^{n_i}-\L^k)$. It is easy to see that this definition does not depend on the choice of the stratification by passing through a common refinement.
\begin{definition}[Pullback]
    For a morphism $f:\Y\to\X$ of algebraic stacks with affine stabilizers. We define a pullback map on generators
    \[
    f^*:\hat{\mathbb{M}}(\X)\to\hat{\mathbb{M}}(\Y),\qquad[Z\to\X]\mapsto[Z\times_\X\Y\to\Y],
    \]
    which is a $\hat{\mathbb{M}}(\mathbb{C})$-algebra homomorphism.
\end{definition}
\begin{definition}[Pushforward]
    Let $f:\Y\to\X$ be a finite type morphism, then we may also define an $\hat{\mathbb{M}}(\mathbb{C})$-algebra homomorphism
    \[
    f_!:\hat{\mathbb{M}}(\Y)\to\hat{\mathbb{M}}(\X),\qquad[Z\to\Y]\mapsto[Z\to\Y\to\X].
    \]
    In particular, if $\X$ itself is of finite type over $\textup{Spec}(\mathbb{C})$, we define the motivic integration to be the pushforward alone structure morphism $\X\to\textup{Spec}(\mathbb{C})$, and denoted by
    \[
    \int_\X:\hat{\mathbb{M}}(\X)\to\hat{\mathbb{M}}(\mathbb{C}).
    \]
\end{definition}
The base change theorem and the projection formula hold in this setting, and they can be checked directly on generators.
\begin{proposition}[{\cite[Theorem 3.5]{Joyce2007}}]Suppose we have a pullback diagram of algebraic stacks with affine stabilizers 
\[
\begin{tikzcd}
	\Y' & \X' \\
	\Y & \X
	\arrow["{f'}", from=1-1, to=1-2]
	\arrow["{g'}", from=1-1, to=2-1]
	\arrow["\lrcorner"{anchor=center, pos=0.125}, draw=none, from=1-1, to=2-2]
	\arrow["g", from=1-2, to=2-2]
	\arrow["f", from=2-1, to=2-2]
\end{tikzcd}
\]
where $f$ is of finite type. 
Then 
\begin{enumerate}
    \item[(a)](Base change formula) We have $g^*\circ f_!=f_!'\circ g^{'*}$,
    \item[(b)](Projection formula)  For any $a\in\hat{\mathbb{M}}(\Y)$ and $b\in\hat{\mathbb{M}}(\X)$ we have
    \[
    f_!(a\cdot f^*(b))=f_!(a)\cdot b.
    \]
\end{enumerate}
\end{proposition}

\subsection{Motives with monodromic action}
Let $\hat{\mu}=\lim\mu_n$ be the projective limit of the groups $\mu_n$ for roots of unity. We consider the ring $\hat{\mathbb{M}}^\textup{mon}(\X)$ of motives with monodromic action on $\X$ in the sense of \cite{Bu2024AMI}. For a variety $Z$, a good $\hat{\mu}$-action is a $\hat{\mu}$-action on $Z$ that factors through a $\mu_n$-action for some $n$  and such that each orbit is contained in an open affine subscheme.
\begin{definition}
The monodromic Grothendieck ring of varieties of $\X$ is defined to be the abelian group
\[
K_\textup{var}^\textup{mon}(\X)=\hat{\bigoplus_{Z\to\X}}\mathbb{Q}\cdot[Z\to\X]^{\hat{\mu}}/\sim,
\]
where we run through all morphisms $Z\to\X$ with $Z$ a variety with a good $\hat{\mu}$-action that is compatible with the trivial $\hat{\mu}$-action on $\X$.
\end{definition}
The ring $K_\textup{var}^\textup{mon}(\X)$ carries a multiplicative structure different from the one induced by fiber products. This product was defined in \cite[Definition 5.3]{BenBassatBravBussiJoyce2015} and is denoted there by $\odot$. Nevertheless, there is a natural ring homomorphism 
\[
\iota:K_\textup{var}(\X)\to K_\textup{var}^\textup{mon}(\X)
\] 
given by equipping the varieties with trivial $\hat{\mu}$-action. Moreover, there is an element 
\[
\L^{\frac{1}{2}}:=1-[\mu_2]^{\hat{\mu}}\in K_\textup{var}^\textup{mon}(\mathbb{C})
\] 
with the nontrivial $\mu_2$-action. It satisfies $(\L^{\frac{1}{2}})^2=\L$. 
\begin{definition}
The monodromic ring of motives on $\X$ is defined to be the localization
\begin{align*}
    \mathbb{M}^\textup{mon}(\X)&=K_\textup{var}^\textup{mon}(\X)\hat{\otimes}_{K_\textup{var}(\mathbb{C})}\left(K_\textup{var}(\mathbb{C})[\L^{-1}]/(\L-1)\textup{-torsion}\right),\\
        \hat{\mathbb{M}}^\textup{mon}(\X)&=K_\textup{var}^\textup{mon}(\X)\hat{\otimes}_{K_\textup{var}(\mathbb{C})}K_\textup{var}(\mathbb{C})[\L^{-1},(\L^k-1)^{-1}]/\approx,
\end{align*}
where the relation $\approx$ is defined in \cite[Definition 5.5]{BenBassatBravBussiJoyce2015}.
\end{definition}
Now we define the motives of double covers.
\begin{definition}
    Let $\mathcal{P}\to\X$ be a principal $\mu_2$-bundle, there is a class
    \[
    \Upsilon(\mathcal{P})=\L^{-\frac{1}{2}}([\X]-[\mathcal{P}]^{\hat{\mu}})\in\hat{\mathbb{M}}^\textup{mon}(\X),
    \]
    where $\hat{\mu}$ acts on $\mathcal{P}$ via $\mu_2$. See \cite[Definitions 5.5, 5.13]{BenBassatBravBussiJoyce2015} for details.
\end{definition}
Note that $\Upsilon$ commutes with pullback and satisfies the relation
\[
\Upsilon(P\otimes_{\mathbb{Z}/2\mathbb{Z}}Q)=\Upsilon(P)\cdot\Upsilon(Q),
\]
for principal $\mathbb{Z}/2\mathbb{Z}$-bundles $P$ and $Q$. 
\subsection{Motivic Behrend Functions}
The Behrend function $\nu_X\colon X\to\mathbb{Z}$ plays a pivotal role in the categorification of Donaldson--Thomas (DT) theory. In his work \cite{Behrend2009DTMicrolocal}, Behrend proved that if a scheme $X$ is equipped with a symmetric obstruction theory, then the degree of its virtual fundamental class can be expressed as the weighted Euler characteristic
\[
\int_{[X]^{\mathrm{vir}}}1 \;=\; \chi\bigl(X,\nu_X\bigr):=\sum_{n\in\mathbb{Z}}n\cdot\chi(\nu_X^{-1}(n)).
\]
From the perspective of DT theory, this result upgrades the theory from counting with a single number to counting with a constructible function: the Behrend function itself already constitutes a first-level categorification of DT invariants. Consequently, in the motivic and cohomological refinements of DT theory, the guiding principle is to further categorify the Behrend function, replacing $\nu_X$ by perverse sheaves, mixed Hodge modules, or motivic classes so that the resulting invariants capture increasingly deep geometric information. We now explain how this categorification is implemented in the motivic setting.

Let $(X,s,o)$ be an oriented d-critical scheme over $\mathbb{C}$. Following \cite[Definition 2.5.3]{Bu2024AMI} and \cite[Theorem 5.7]{BenBassatBravBussiJoyce2015} we can define the motivic Behrend function $\nu_{X,s,o}^\textup{mot}$ based on \cite[Theorem 5.10]{BussiJoyceMeinhardt2019} .
\begin{definition}\label{motivicBF}
    Let $(X,s,o)$ be an oriented d-critical scheme. Its motivic Behrend function $\nu_X^\textup{mot}\in\hat{\mathbb{M}}^\textup{mon}(X)$ is defined uniquely by the following property:

    \begin{itemize}
        \item For any critical chart $i:\textup{Crit}(f)\hookrightarrow X$ where $f:U\to\mathbb{A}^1$ is a regular function on a smooth scheme $U$, there is an identity
    \[
    i^*\nu_X^\textup{mot}=-\L^{-\frac{\textup{dim}U}{2}}\cdot\Phi_f([U])\cdot\Upsilon(i^*(K_X^{\frac{1}{2}})\otimes K^{-1}_U|_{\textup{Crit}(f)^\textup{red
    }}).
    \]
    Here $\Phi_f$ is the motivic vanishing cycle and we identify principal $\mu_2$-bundles with line bundles which square to trivial
    \end{itemize} 
\end{definition}
Roughly speaking, the motivic Behrend function is defined locally via vanishing cycles together with local isomorphism data determined by the chosen orientation.

Note that the motivic Behrend function depends on the orientation and the d-critical structure. We will simply denote it by $\nu_X^\textup{mot}$ if there is no ambiguity. The definition can be extended to algebraic stacks with orientation and d-critical structure. Note that for any smooth morphism $f:Y\to\X$, there is an induced d-critical structure and an induced orientation on $Y$ by Lemma \ref{inducedorientation}.
\begin{theorem}[{\cite[Theorem 2.5.4]{Bu2024AMI}}]
        Let \((\mathcal{X},s,o)\) be an oriented d-critical stack that is Nisnevich-locally a quotient stack.  
Then there exists a unique class
\[
\nu_{\mathcal{X}}^{\mathrm{mot}}
=
\nu_{\mathcal{X},s,o}^{\mathrm{mot}}
\;\in\;
\widehat{\mathbb{M}}^{\mathrm{mon}}(\mathcal{X}),
\]
called the motivic Behrend function of \(\mathcal{X}\), characterized by the following property:  
for any variety \(Y\) and any smooth morphism
\(f:Y\to\mathcal{X}\) of relative dimension \(d\), we have
\[
f^{*}\bigl(\nu_{\mathcal{X}}^{\mathrm{mot}}\bigr)
=\;
\mathbb{L}^{d/2}\,\nu_{Y}^{\mathrm{mot}}
\;\in\;
\widehat{\mathbb{M}}^{\mathrm{mon}}(Y),
\]
where $Y$ is equipped with the induced oriented d-critical structure. 
\end{theorem}
For an oriented $(-1)$-shifted symplectic stack $(\dX,\omega_\dX,o)$, we write $\nu_{\dX}^\textup{mot}:=\nu_\X^\textup{mot}$, where $\X=\dX_\textup{cl}$ is its classical truncation equipped with the induced orientation and d-critical structure.
\begin{theorem}[{\cite[Theorem 2.5.5]{Bu2024AMI}}, {\cite[Theorem 5.14]{BenBassatBravBussiJoyce2015}}]
    Let $\X$, $\Y$ be oriented d-critical stacks that are Nisnevich locally quotient stacks, and let $f:\Y\to\X$ be a smooth morphism of relative dimension $d$ which is compatible with the orientation and d-critical structures. Then there is an identity
    \[
    f^*(\nu_\X^\textup{mot})=\L^{d/2}\cdot \nu_\Y^\textup{mot}\in\hat{\mathbb{M}}^\textup{mon}(\Y).
    \]
\end{theorem}
The following theorem gives a formula for motivic Behrend function on the product, though it is a direct consequence of the Thom--Sebastiani Theorem for motivic vanishing cycle, we provide a proof here. In the proof we will use the Thom--Sebastiani Theorem for motivic vanishing cycle \cite[Theorem 2.4]{BussiJoyceMeinhardt2019}.
\begin{theorem}[Thom--Sebastiani Theorem for Motivic Behrend function]\label{TStheorem}\leavevmode

\noindent Let $(\X,s,o_\X)$, $(\Y,t,o_\Y)$ be oriented d-critical stacks. Consider the induced oriented d-critical structure on the product $(\X\times\Y,s\boxplus t, o_\X\boxtimes o_\Y)$. Then we have
    \[
    \nu_\X^\textup{mot}\boxtimes \nu_\Y^\textup{mot}=\nu_{\X\times\Y}^{\textup{mot}}\in\hat{\mathbb{M}}^\textup{mon}(\X\times\Y).
    \]
\end{theorem}
\begin{proof}
We first prove the theorem when $\X$ and $\Y$ are algebraic spaces. Since algebraic spaces are Nisnevich locally covered by schemes and our motives satisfy Nisnevich descent \cite[Theorem 2.2.3]{Bu2024AMI}, we may consider the case when $\X=X$, $\Y=Y$ are schemes. 

We now verify this on d-critical charts. Let $i:\textup{Crit}(f)\hookrightarrow X$, $j:\textup{Crit}(g)\hookrightarrow Y$ be d-critical chart, where $f:U\to\mathbb{A}^1$, $g:V\to\mathbb{A}^1$ are regular functions. Then we consider the d-critical chart $i\times j:\textup{Crit}(f)\times\textup{Crit}(g)\hookrightarrow X\times Y$.

By the definition of motivic Behrend function (\ref{motivicBF}) we have
\[
(i\times j)^*\nu_{X\times Y}^\textup{mot}=\L^{-(\dim U\times V)/2}\Phi_{f\boxplus g}([U\times V])\Upsilon((i\times j)^*(K_{X\times Y}^{\frac{1}{2}})\otimes K_{U\times V}^{-1}|_{\textup{Crit}(f\boxplus g)^{\textup{red}}})
\]
The Thom-Sebastiani Theorem for motivic vanishing cycles \cite[Theorem 2.4]{BussiJoyceMeinhardt2019} yields
\[
\Phi_f(U)\boxtimes\Phi_g(V)=\Phi_{f\boxplus g}(U\times V)
\]
and the product orientation on $X\times Y$ is defined by $K_{X\times Y}^{\frac{1}{2}}=K_X^\frac{1}{2}\boxtimes K_Y^\frac{1}{2}$. Thus
\begin{align*}
    &\L^{-(\dim U\times V)/2}\Phi_{f\boxplus g}([U\times V])\Upsilon((i\times j)^*(K_{X\times Y}^{\frac{1}{2}})\otimes K_{U\times V}^{-1}|_{\textup{Crit}(f\boxplus g)^{\textup{red}}})\\
    =&\L^{-\dim U/2}\Phi_f([U])\Upsilon(i^*(K^\frac{1}{2}_X)\otimes K^{-1}_X|_{\textup{Crit}(f)^\textup{red}})\boxtimes\L^{-\dim V/2}\Phi_g([V])\Upsilon(j^*(K^\frac{1}{2}_Y)\otimes K^{-1}_Y|_{\textup{Crit}(g)^\textup{red}})\\
    =&i^*\nu_X^\textup{mot}\boxtimes j^*\nu_Y^\textup{mot}.
\end{align*}
Now we consider the general case, by assumption the stacks $\X$ and $\Y$ are Nisnevich locally quotient stacks. We consider the case $\X=[X/\textup{GL}(n)]$ and $\Y=[Y/\textup{GL}(m)]$ when $X$ and $Y$ are algebraic space with orientations and d-critical structures. Let $p:X\to[X/\textup{GL}(n)]$, $q:Y\to[Y/\textup{GL}(m)]$ be quotient maps. By the construction of motivic Behrend function on stacks,
\[
\nu_\X^\textup{mot}=[\textup{GL}(n)]^{-1}p_!(\L^{n^2/2}\cdot \nu_X^{\textup{mot}}), \nu_\Y^\textup{mot}=[\textup{GL}(m)]^{-1}q_!(\L^{m^2/2}\cdot \nu_Y^{\textup{mot}})
\]
We finally get
\begin{align*}
    \nu_\X^\textup{mot}\boxtimes \nu_\Y^\textup{mot}
    &=[\textup{GL}(n)\times\textup{GL}(m)]^{-1}(p\times q)_!(\mathbb{L}^{m^2+n^2/2}\cdot\nu_X^\textup{mot}\boxtimes \nu_Y^{\textup{mot}})\\
    &=[\textup{GL}(n+m)]^{-1}\cdot(p\times q)_!(\L^{(m+n)^2/2}\nu_{X\times Y}^\textup{mot})=\nu_{\X\times \Y}^\textup{mot}.
\end{align*}
\end{proof}
\subsection{Graded and Filtered objects}
We introduce the notion of graded objects and filtered objects following \cite{HalpernLeistner2014}.
\begin{definition}[Mapping stack]
    Let $\mathcal{X}, \mathcal{Y}$ be stacks over some site, we define the mapping stack 
    \begin{align*}
            \underline{\textup{Map}}(\mathcal{Y},\mathcal{X}):&(\textup{Sch}/\mathbb{C})\to(\textup{Gpd})\\
            & T\mapsto\textup{Map}(\mathcal{Y}\times T,\mathcal{X})
    \end{align*}
    
\end{definition}

\begin{definition}
    Let $\mathcal{X}$ be an algebraic stack over $\mathbb{C}$. The stack of graded objects in $\mathcal{X}$ is defined as the mapping stack over $\textup{Spec}(\mathbb{C})$
\begin{align*}
        &\textup{Grad}(\mathcal{X}):=\underline{\textup{Map}}([*/\mathbb{G}_m],\mathcal{X})\\
        &\textup{Filt}(\mathcal{X}):= \underline{\textup{Map}}([\mathbb{A}^1/\mathbb{G}_m],\mathcal{X})
\end{align*}
where, $\mathbb{G}_m$ acts on $\mathbb{A}^1$ by weight $-1$.
\end{definition}
These stacks are again algebraic stack by the following porposition.
\begin{proposition}[{\cite[Proposition 1.1.2]{HalpernLeistner2014}}]
    Let $\mathcal{X}$ be an algebraic stack over $\mathbb{C}$. Then $\textup{Grad}(\mathcal{X})$ and $\textup{Filt}(\mathcal{X})$ are also algebraic. Moreover, if $\mathcal{X}$ has affine
stabilizers, then so do Grad($\mathcal{X}$) and Filt($\mathcal{X}$).
\end{proposition}

We now consider the following diagram 
\[\begin{tikzcd}
	{[*/\mathbb{G}_m]} & {[\mathbb{A}^1/\mathbb{G}_m]} & *
	\arrow["\iota"', shift right, from=1-1, to=1-2]
	\arrow["q"{description}, curve={height=22pt}, from=1-3, to=1-1,shift right]
	\arrow["{\textup{pr}}"', shift right, from=1-2, to=1-1]
	\arrow["0"', shift right, from=1-3, to=1-2]
	\arrow["1", shift left, from=1-3, to=1-2]
\end{tikzcd}\]
where $q$ is the quotient map, $0,1$ are the compositions $0,1:*\to\mathbb{A}^1\to\mathbb{A}^1/\mathbb{G}_m$, $\iota$ is the inclusion map and $\textup{pr}$ is the projection map. These morphisms induced morphisms between $\textup{Grad}(\mathcal{X})$, $\textup{Filt}(\mathcal{X})$ and $\mathcal{X}$ via pullback
\[\begin{tikzcd}
	{\textup{Grad}(\mathcal{X})} & {\textup{Filt}(\mathcal{X})} & \mathcal{X}
	\arrow["{\textup{sf}}", shift left, from=1-1, to=1-2]
	\arrow["{\textup{tot}}"{description}, curve={height=-24pt}, from=1-1, to=1-3,shift left]
	\arrow["{\textup{gr}}", shift left, from=1-2, to=1-1]
	\arrow["{\textup{ev}_0}", shift left, from=1-2, to=1-3]
	\arrow["{\textup{ev}_1}"', shift right, from=1-2, to=1-3]
\end{tikzcd}\]

The stack of graded objects \( \mathrm{Grad}(\dX) \) inherits a natural orientation and a \( (-1) \)-shifted symplectic structure from \( \dX \).
\begin{theorem}[{\cite[Theorem 3.1.6]{Bu2024AMI}}]
        The derived stack $\textup{Grad}(\dX)$ has a $(-1)$ shifted symplectic structure given by $\textup{tot}^\star\omega_\dX$ and an orientation $o_{\textup{Grad}(\dX)}$ induced from $\dX$.
\end{theorem}
Let $f:\Y\to\X$ be a morphism of algebraic stacks, we may consider the induced morphism on their stacks of graded and filtered objects via composition with $f$.
\begin{align*}
    \textup{Grad}(f)&:=(f\circ -):\textup{Grad}(\Y)\longrightarrow\textup{Grad}(\X)\\
    \textup{Filt}(f)&:=(f\circ -):\textup{Filt}(\Y)\longrightarrow\textup{Filt}(\X).
\end{align*}
These maps inherit properties from the map $f$.  The following result will be needed later in this paper.
\begin{corollary}\label{restrictionlemma}
    Let $f:\Y\to\X$ be an open immersion. We have the following commutative diagram:    
\[\begin{tikzcd}
	\textup{Grad}(\Y) & \textup{Filt}(\Y) & \Y \\
	\textup{Grad}(\X) & \textup{Filt}(\X) & \X
	\arrow[from=1-1, to=2-1]
	\arrow["{\textup{gr}}"',from=1-2, to=1-1]
	\arrow["{\textup{ev}_1}", from=1-2, to=1-3]
	\arrow[from=1-2, to=2-2]
	\arrow[from=1-3, to=2-3]
	\arrow["{\textup{gr}}"',from=2-2, to=2-1]
	\arrow["{\textup{ev}_1}", from=2-2, to=2-3]
    \arrow["\lrcorner"{anchor=center, pos=0.125, rotate=-90}, draw=none, from=1-2, to=2-1]
\end{tikzcd}\]
where all the vertical maps are open immersions.
\end{corollary}
\begin{proof}
    This follows from \cite[Section 1]{HalpernLeistner2014} and \cite[Lemma 3.2.4.]{Bu2024AMI}.
\end{proof}

\subsubsection{Derived version}
    The stacks of grade objects and filtered objects has a natural derived enhancement. For derived stacks $\mathfrak{X}, \mathfrak{Y}$, we defined derived mapping stacks
    \begin{align*}
        \underline{\textup{dMap}}(\mathfrak{Y},\mathfrak{X}):(\textup{dSch}&/\mathbb{C})\to\S\\
        &T\mapsto\textup{Map}(\mathfrak{Y}\times T,\mathfrak{X})
    \end{align*}
    from the $\infty$-category of derived shcemes to the $\infty$-category of spaces.

   For a derived stack $\mathfrak{X}$
   \begin{align*}
   \textup{Filt}(\mathfrak{X}):=\underline{\textup{dMap}}([\mathbb{A}^1/\mathbb{G}_m],\mathfrak{X}),\
   \textup{Grad}(\mathfrak{X}):=\underline{\textup{dMap}}([*/\mathbb{G}_m],\mathfrak{X}).
   \end{align*}
   Note that given dervied stacks $\mathfrak{X},\mathfrak{Y}$. The classical truncation of the derived mapping stack is not necessary the same as the mapping stack of their classical truncations. However, we have the following result
\begin{proposition}[{\cite[Theorem 1.2.1]{HalpernLeistner2014}}]
    Let $\mathfrak{X}$ be a derived algebraic stack, then 
    \[
    \textup{Filt}(\mathfrak{X})_{\textup{cl}}=\textup{Filt}(\mathfrak{X}_{\textup{cl}}), \textup{Grad}(\mathfrak{X})_{\textup{cl}}=\textup{Grad}(\mathfrak{X}_{\textup{cl}})
    \]
\end{proposition}
\subsection{Motivic Integral Identity}
The motivic integral identity was first conjectured by Kontsevich and Soibelman in \cite[Conjecture 4.]{KS2008} as a key ingredient in the formulation of their motivic wall-crossing formula. The version stated below is a generalization that applies to oriented $(-1)$-shifted symplectic stacks, formulated in terms of the stacks of graded and filtered objects.

\begin{theorem}[{\cite[Theorem 4.2.2]{Bu2024AMI}}]\label{mii}
    Let $\dX$ be an oriented $(-1)$-shifted symplectic stack over $\mathbb{C}$ with classical truncation $\X=\dX_{\textup{cl}}$. Assume that $\X$ is Nisnevich locally fundamental. Consider the $(-1)$-shifted Lagrangian correspondence
    \[
    \Grad(\dX)\overset{\textup{gr}}{\leftarrow}\Filt(\dX)\overset{\textup{ev}_1}{\to}\dX
    \]
    Then we have the identity
    \[
    \textup{gr}_!\circ\textup{ev}_1^*(\nu_{\dX}^{\textup{mot}})=\mathbb{L}^{\textup{vdim}\Filt(\dX)/2}\nu_{\Grad(\dX)}^\textup{mot}
    \]
    in $\hat{\mathbb{M}}^{\textup{mon}}(\Grad(\X))$, where the $\textup{vdim}\Filt(\dX)$ denote the virtual dimension of $\Filt(\dX)$, seen as a function $\pi_0(\Filt(\dX))\cong\pi_0(\Grad(\dX))\to\mathbb{Z}$.
\end{theorem}
\begin{remark}
The proof of the motivic integral identity relies on explicit local computations and a gluing argument that uses Nisnevich-local fundamental covers. Consequently, the existence of such covers is essential. In our setting, we invoke Corollary \ref{maincoro} to guarantee the required Nisnevich-local structure.
\end{remark}
At some point we will need to compute the virtual dimension of the stack
$\Filt(\dX)$, so we record the following useful lemma of
Halpern-Leistner.

\begin{lemma}[{\cite[Lemma 1.2.3]{HalpernLeistner2014}}]\label{virtualdimensionformula}
    Let $\dX$ be a derived algebraic stack locally of finite presentation
    over $\mathbb{C}$. Then the tangent complex of $\Filt(\dX)$ is
    described in terms of $\mathbb{T}_\dX$ by
    \[
        \textup{sf}^*(\mathbb{T}_{\Filt(\dX)})
        \;\cong\;
        \textup{tot}^*(\mathbb{T}_\dX)^{\geq 0},
    \]
    where the superscript $(\geq 0)$ denotes the nonnegative weight part
    with respect to the natural $\mathbb{G}_m$-action.
\end{lemma}


\section{Motivic Donaldson--Thomas invariants}
\subsection{Donaldson--Thomas type invariants}
Fix a numerical type $v=(-n,-\beta,1)\in\Gamma^1$. We now define motivic DT and PT invariants, respectively. Recall that there is an induced orientation and d-critical structure on the stack $\M^{ss}$. Thus motivic Behrend functions are well-defined on both $\M_{\text{DT}}^{v}$ and $\M_{\text{PT}}^{v}$.
\begin{definition}\label{Definitionofdtinvariants}
    We define
    \[
    \mathrm{DT}^{\mathrm{mot}}_{n,\beta}(X)=\int_{\mathcal{M}_{\mathrm{DT}}^v}(\L^{\frac{1}{2}}-\L^{-\frac{1}{2}})\nu_{\mathcal{M}_{\mathrm{DT}}^v}^\mathrm{mot}\in\hat{\mathbb{M}}^{\mathrm{mon}}(\mathbb{C})
    \]
    \[
        \mathrm{PT}^{\mathrm{mot}}_{n,\beta}(X)=\int_{\mathcal{M}_{\mathrm{PT}}^v}(\L^{\frac{1}{2}}-\L^{-\frac{1}{2}})\nu_{\mathcal{M}_{\mathrm{PT}}^v}^\mathrm{mot}\in\hat{\mathbb{M}}^{\mathrm{mon}}(\mathbb{C}).
    \]
\end{definition}
\begin{remark}
Applying the Euler characteristic morphism recovers the numerical DT and PT invariants from their motivic counterparts. Indeed, by Theorem~\ref{M=I/G} we have 
\[
\M_{\mathrm{DT}}^v \cong [I_n(X,\beta)/\mathbb{G}_m].
\]
Let $\pi : I_n(X,\beta) \to [I_n(X,\beta)/\mathbb{G}_m]$ be the quotient map. 
This is a smooth morphism of relative dimension $1$, hence
\[
\pi^*\left(\nu_{\M_{\mathrm{DT}}^v}^{\mathrm{mot}}\right)
  = \L^{1/2}\cdot \nu_{I_n(X,\beta)}^{\mathrm{mot}}.
\]

Since $\pi$ is a principal $\mathbb{G}_m$-bundle, for any class 
$[Y \to \M_{\mathrm{DT}}^v]$ we have the motivic identity
\[
\int_{I_n(X,\beta)} \pi^*[Y\to\M_{\text{DT}}^v]
   = (\L-1)\int_{\M_{\mathrm{DT}}^v} [Y\to\M_{\text{DT}}^v].
\]
Combining the two relations, we obtain
\[
\mathrm{DT}^{\mathrm{mot}}_{n,\beta}(X)
   = \int_{\M_{\mathrm{DT}}^v}
      (\L^{1/2}-\L^{-1/2})\,
      \nu_{\M_{\mathrm{DT}}^v}^{\mathrm{mot}}
   = \int_{I_n(X,\beta)}
      \nu_{I_n(X,\beta)}^{\mathrm{mot}}.
\]
Applying Euler characteristic morphism we obtain
\[
\chi\left(\mathrm{DT}^{\mathrm{mot}}_{n,\beta}(X)\right)=\int_{I_n(X,\beta)}\nu_{I_n(X,\beta)}=\mathrm{DT}_{n,\beta}(X).
\]
\end{remark}

\begin{theorem}\label{main}
    We have the following identity: 
    \[
    \sum_{n,\beta}\mathrm{DT}^{\mathrm{mot}}_{n,\beta}(X)q^nx^\beta=\left(\sum_{n\geq0}\mathrm{DT}^{\mathrm{mot}}_{n,0}(X)q^n\right)\left(\sum_{n,\beta}\mathrm{PT}^{\mathrm{mot}}_{n,\beta}(X)q^nx^\beta\right)
    \]
    in $\hat{\mathbb{M}}^{\mathrm{mon}}(\mathbb{C})[[\Gamma^1]]$.
\end{theorem}

\subsection{Motivic Hall Algebra}
We recall the construction of the motivic Hall algebra, originally defined by Joyce in \cite{Joyce2006confiI,Joyce2007confiII}, and further developed in recent formulations by Bu, Kinjo, and Núñez \cite{Bu2025OrthosymplecticDT,IntrinsicDTII}. See also \cite{Bridgeland2010} for a nice introduction. Following \cite[Section 3]{toda2020hall}, we define the following Hall algebra.
Define
\[
\mathcal{H}(\A_X)=\bigoplus_{v\in\Gamma^1\cup\Gamma^0}\mathbb{M}(\O bj^v(\A_X)).,
\]
as a $\mathbb{M}(\mathbb{C})$-module. 
We consider the $\mathbb{M}(\mathbb{C})$-submodule 
\[
\mathcal{H}^0:=\bigoplus_{v\in\Gamma^0}\mathbb{M}(\O bj^v(\A_X)).
\]
\noindent The $\mathbb{M}(\mathbb{C})$-module $\mathcal{H}(\A_X)$ carries a  $\mathcal{H}^0$-bimodule structure defined as follows:

    Given elements$a=[\X\overset{f}{\to}\mathcal{O}bj^v(\mathcal{A}_X)]\in\mathcal{H}(\A_X)$ and $b=[\Y\overset{g}{\to}\mathcal{O}bj^w(\mathcal{A}_X)]\in\mathcal{H}^0$. Consider the following diagram
    \begin{equation}\label{diag:star}
  \begin{tikzcd}[row sep=large,column sep=large]
    \mathcal{Z} \arrow[r] \arrow[d] & 
    \mathcal{E}x^{(v,w)}(\mathcal{A}_X) \arrow[r] \arrow[d] &
    \mathcal{O}bj^{v+w}(\mathcal{A}_X) \\
    \X\times \Y \arrow[r,"{(f,g)}"] &
    \mathcal{O}bj^v(\mathcal{A}_X)\times\mathcal{O}bj^{w}(\mathcal{A}_X)
  \end{tikzcd}
  \tag{$*$}
\end{equation}

where $\mathcal{E}x^{(v,w)}(\A_X)$ is the stack which parametrizes exact sequences 
    \[
    \begin{tikzcd}
	0 & E & F & G & 0
	\arrow[from=1-1, to=1-2]
	\arrow[from=1-2, to=1-3]
	\arrow[from=1-3, to=1-4]
	\arrow[from=1-4, to=1-5]
    \end{tikzcd}
    \]
    where $E,G$ are objects in $\A_X$ such that $\mathrm{cl}(E)\in\Gamma^1$, $\mathrm{cl}(G)\in\Gamma^0$.  
    The morphism 
    \[
    \mathcal{E}x^{(v,w)}(\A_X)\to \O bj^v(\A_X)\times \O bj^w(\A_X)
    \]
    is given by sending an exact sequence as above to the pair $(E,G)$ and the morphism $\mathcal{E}x^{(v,w)}(\A_X)\to\O bj^{v+w}(\A_X)$ sends the above exact sequence to $F$. We define the product as
\[
a*b=[\X \xrightarrow{f} \O bj^v(\mathcal{A}_X)] * [\Y \xrightarrow{g} \O bj^{w}(\mathcal{A}_X)] := [\mathcal{Z} \xrightarrow{h} \O bj^{v+w}(\mathcal{A}_X)].
\]
The product $b*a$ is defined in the same way. We call this the Hall algebra product.
It has been shown that this process can be formulated using stacks of graded and filtered objects via the following pullback diagram.
\begin{lemma}[{\cite[Proposition 8.26]{KinjoParkSafronov2024}}]There is a pullback diagram such that the horizontal morphisms are open immersions
\[\begin{tikzcd}
	\mathcal{E}x^{(v,w)}(\O bj(\A_X)) & \mathrm{Filt}(\O bj(\A_X)) \\
	\O bj^v(\A_X)\times \O bj^w(\A_X) & \mathrm{Grad}(\O bj(\A_X))
	\arrow[from=1-1, to=1-2]
	\arrow[from=1-1, to=2-1]
	\arrow[from=1-2, to=2-2]
	\arrow["{\widetilde{\Phi}_2}", from=2-1, to=2-2]
\end{tikzcd},
\]
where the morphism $\widetilde{\Phi}_2$ is given by inclusion of graded objects in weights $0$ and $-1$.
\end{lemma}
\begin{proof}
    Let $\mathcal{C}$ be the perfect complexes on $X$ with trivial determinant and take the poset $P=\{0<1<2\}$. First apply \cite[Proposition 8.26]{KinjoParkSafronov2024} to $\mathcal{C}$ and $P$, and then restrict the resulting diagram to the open substack $\M_0$ of universally gluable objects,  this gives an open immersion of diagrams.
    \[\begin{tikzcd}
	\M_0^{2-\mathrm{filt}} & \mathrm{Filt}(\M_0) \\
	\M_0\times\M_0 & \mathrm{Grad}(\M_0)
	\arrow[from=1-1, to=1-2]
	\arrow[from=1-1, to=2-1]
	\arrow[from=1-2, to=2-2]
	\arrow["{\widetilde{\Phi}_2}", from=2-1, to=2-2]
    \end{tikzcd},
    \]
    where $\M_0^{2-\mathrm{filt}}$ denotes the moduli stack parametrizes pairs $E\subset F\in\M_0$. Further restricting on the open substack $\O bj(\A_X)$ we obtain
    \[\begin{tikzcd}
	\mathcal{E}x(\O bj(\A_X)) & \mathrm{Filt}(\O bj(\A_X)) \\
	\O bj(\A_X)\times \O bj(\A_X) & \mathrm{Grad}(\O bj(\A_X))
	\arrow[from=1-1, to=1-2]
	\arrow[from=1-1, to=2-1]
	\arrow[from=1-2, to=2-2]
	\arrow["{\widetilde{\Phi}_2}", from=2-1, to=2-2]
    \end{tikzcd}.
    \]
    The results then follows from the restriction
    \[\begin{tikzcd}
	\mathcal{E}x^{(v,w)}(\O bj(\A_X)) & \mathcal{E}x(\O bj(\A_X)) \\
	\O bj^v(\A_X)\times \O bj^w(\A_X) & \O bj(\A_X)\times \O bj(\A_X)
	\arrow[from=1-1, to=1-2]
	\arrow[from=1-1, to=2-1]
	\arrow[from=1-2, to=2-2]
	\arrow[from=2-1, to=2-2]
    \end{tikzcd}.
    \]
\end{proof}
\noindent On the other hand, we consider the following abelian group
\[
\hat{\mathbb{M}}^{\mathrm{mon}}(\mathbb{C})[\Gamma]:=\bigoplus_{v\in\Gamma}\hat{\mathbb{M}}^{\mathrm{mon}}(\mathbb{C})\cdot x^v
\]
where $x^v$s are formal variables. We define $*$-product on this abelian group as follows, for $x^v$ and $x^w$ we define
\[
x^v * x^{w} := \L^{\chi(v, w)/2} x^{v + w}.
\]
Extend $\hat{\mathbb{M}}^{\mathrm{mon}}$-linearly for general elements, i.e. 
\[
\sum_v a_v x^v * \sum_w b_w x^w:=\sum_{u} \left(\sum_{v+w=u}\L^{\chi(v,w)/2}a_v b_w\right) x^{u}
\]
We will also consider the usual product $\cdot$ on $\hat{\mathbb{M}}^{\mathrm{mon}}(\mathbb{C})[\Gamma]$ i.e. 
$x^v\cdot x^w:= x^{v+w}.$

\subsection{Motivic Integration Map}
The integration map was used in the proof of the numerical DT/PT correspondence, as in \cite{Bridgeland2011}. We now aim to upgrade this integration map to the motivic level. We define the following map:
\begin{align*}
    I: \mathcal{H}(\A_X) &\longrightarrow \hat{\mathbb{M} }^{\mathrm{mon}}(\mathbb{C})[\Gamma], \\
    [\X \xrightarrow{f} \mathcal{O} bj^v(\mathcal{A}_X)] &\mapsto \left( \int_\X f^* \nu^{\mathrm{mot}}_{\mathcal{O} bj(\mathcal{A}_X)} \right) x^v,
\end{align*}
and extended $\hat{\mathbb{M} }^{\mathrm{mon}}(\mathbb{C})$-linearly.

As in the numerical DT/PT correspondence
\cite{Bridgeland2011, Toda2010}, one would like the motivic integration
map to be compatible with the Hall algebra product.  
To establish such a compatibility at the motivic level, the essential
ingredient is Bu's motivic integral identity.  

However, to apply the motivic integral identity, one needs the stack
$\mathcal{O}bj^{v}(\mathcal{A}_X)$ to admit a Nisnevich locally
fundamental cover.  This is not known in general, and therefore the
standard approach does not immediately apply in our situation.  

On the other hand, thanks to the existence of good moduli spaces for
$\mathcal{M}^{ss}$ and the structure theorem for Artin stacks admitting
good moduli spaces, the semistable locus $\mathcal{M}^{ss}$ does
admit a Nisnevich locally fundamental cover.  This allows us to establish
the motivic integral identity on $\mathcal{M}^{ss}$.

Consequently, from now on we restrict to the $\mathcal{H}^0$-submodule
$\mathcal{H}^{ss}\subset\mathcal{H}(\A_X)$ generated by elements 
\[
a=[\mathcal{X}\xrightarrow{f}\mathcal{O}bj^{v}(\mathcal{A}_X)]
\in \mathcal{H}(\mathcal{A}_X)
\]
such that the morphism $f$ factors through the semistable locus
$\mathcal{M}^{ss,v}$ for some class $v\in\Gamma^0\cup\Gamma^1$.

\begin{theorem}\label{homomorphism}
The restriction map $I:\mathcal{H}^{ss}\to\hat{\mathbb{M} }^{\mathrm{mon}}(\mathbb{C})[\Gamma]$, is an $\mathcal{H}^0$-module homomorphism. In other words, for any $a\in\mathcal{H}^{ss}$, $b\in\mathcal{H}^0$ we have
\[
I(a)*I(b)=I(a*b),\qquad I(b)*I(a)=I(b*a)
\]
\end{theorem}
\begin{proof}To show \(I(a)*I(b)=I(a*b)\), for each cartesian diagram \eqref{diag:star} it suffices to prove
\[
    \int_{\mathcal{Z}} h^{*}\,\nu^{mot}_{\O bj(\mathcal{A}_X)}
    \;=\;
    \L^{\chi(v,v_0)/2}\,
    \Bigl(\int_{\mathcal{X}} f^{*}\,\nu^{mot}_{\O bj(\mathcal{A}_X)}\Bigr)\,
    \Bigl(\int_{\mathcal{Y}} g^{*}\,\nu^{mot}_{\O bj(\mathcal{A}_X)}\Bigr).
\]
Consider the diagram
\[
\begin{tikzcd} \mathcal{Z} & \mathcal{E}x^{(v,v_0)}(\mathcal{A}_X) & \mathrm{Filt}(\mathcal{O} bj(\mathcal{A}_X)) & \mathcal{O} bj(\mathcal{A}_X) \\ \X \times \Y & \mathcal{O} bj^v(\mathcal{A}_X) \times \mathcal{O} bj^{v_0}(\mathcal{A}_X) & \mathrm{Grad}(\mathcal{O} bj(\mathcal{A}_X)) \arrow[from=1-1, to=1-2] \arrow["h", curve={height=-20pt}, from=1-1, to=1-4] \arrow[from=1-1, to=2-1] \arrow[from=1-2, to=1-3] \arrow[from=1-2, to=2-2] \arrow["\mathrm{ev}_1", from=1-3, to=1-4] \arrow["\mathrm{gr}", from=1-3, to=2-3] \arrow["{(f,g)}", from=2-1, to=2-2] \arrow["\widetilde{\Phi}_2", from=2-2, to=2-3] \arrow["\lrcorner"{anchor=center, pos=0.125}, draw=none, from=1-1, to=2-2] \arrow["\lrcorner"{anchor=center, pos=0.125}, draw=none, from=1-2, to=2-3] \end{tikzcd}
\]
and, by Corollary~\ref{restrictionlemma}, the restriction diagram
\[
\begin{tikzcd}
    \Grad(\M^{ss}) & \Filt(\M^{ss}) \ar[l,swap,"\mathrm{gr}"] \ar[r,"\mathrm{ev}_1"] & \M^{ss} \\
    \Grad(\O bj(\mathcal{A}_X)) \ar[u] & \Filt(\O bj(\mathcal{A}_X)) \ar[l,swap,"\mathrm{gr}"] \ar[u] \ar[r,"\mathrm{ev}_1"] & \O bj(\mathcal{A}_X) \ar[u]
\end{tikzcd}
\]
Note that, by Lemma~\ref{directsumofwithpoint}, the maps \(h\) and \(\widetilde{\Phi}_2\circ(f,g)\) factor through \(\Filt(\M^{ss})\) and \(\Grad(\M^{ss})\), respectively. Abusing notation, we again write
\[
h:\mathcal{Z}\to \Filt(\M^{ss})\xrightarrow{\ \mathrm{ev}_1\ }\M^{ss},
\qquad
(f,g):\mathcal{X}\times\mathcal{Y}\to \Grad(\M^{ss}).
\]
Hence
\[
\int_{\mathcal{Z}} h^{*}\nu^{mot}_{\O bj(\mathcal{A}_X)}
=\int_{\mathcal{Z}} h^{*}\nu^{mot}_{\M^{ss}},
\qquad
\Bigl(\int_{\mathcal{X}} f^{*}\nu^{mot}_{\O bj(\mathcal{A}_X)}\Bigr)\!
\Bigl(\int_{\mathcal{Y}} g^{*}\nu^{mot}_{\O bj(\mathcal{A}_X)}\Bigr)
=\Bigl(\int_{\mathcal{X}} f^{*}\nu^{mot}_{\M^{ss}}\Bigr)\!
\Bigl(\int_{\mathcal{Y}} g^{*}\nu^{mot}_{\M^{ss}}\Bigr).
\]

Since \(\M^{ss}\) is Nisnevich-locally fundamental (Corollary~\ref{maincoro}), we can apply the motivic integral identity to obtain
\[
\int_{\mathcal{Z}} h^{*}\nu^{mot}_{\M^{ss}}
= \int_{\mathcal{X}\times\mathcal{Y}} (f,g)^{*}\,\widetilde{\Phi}_2^{*}\,\mathrm{gr}_{!}\,\mathrm{ev}_1^{*}\nu^{mot}_{\M^{ss}}
= \L^{d/2}\int_{\mathcal{X}\times\mathcal{Y}} (f,g)^{*}\,\widetilde{\Phi}_2^{*}\,\nu_{\Grad(\M^{ss})},
\]
where \(d\) is the virtual dimension of \(\Filt(\O bj(\mathcal{A}_X))\) at the \(\mathbb{C}\)-point \(\mathrm{sf}\big(\widetilde{\Phi}_2(E,F)\big)\) with
\((E,F)\in \M^{ss,v}\times \M^{ss,v_0}\).

By Lemma~\ref{virtualdimensionformula},
\[
d \;=\; \chi\!\bigl(\L_{\Filt(\M_0)}|_{\mathrm{sf}(E\oplus F)}\bigr)
= \chi\!\bigl((\L_{\M_0}\!\mid_{E\oplus F})^{\le 0}\bigr)
= \chi\!\bigl(R\!\Hom(E\oplus F,E\oplus F)_0[2]^{\le 0}\bigr),
\]
using that the tangent complex at a point \(P\) of \(\M_0\) is the traceless derived endomorphism
(see \cite[§5]{SchurgToenVezzosi2015}).
The traceless part sits in an exact triangle
\[
R\!\Hom(E\oplus F,E\oplus F)_0[1] \longrightarrow
R\!\Hom(E\oplus F,E\oplus F)[1] \xrightarrow{\ \text{tr}\ } R\Gamma(X,\O_X)[1] \longrightarrow
\]
and the Calabi--Yau condition implies
\[
\chi\!\bigl(R\!\Hom(E\oplus F,E\oplus F)_0[1]\bigr)
= \chi\!\bigl(R\!\Hom(E\oplus F,E\oplus F)[1]\bigr) + \chi(\O_X) \;=\; 0.
\]
Therefore
\begin{align*} \chi(R\Hom(E\oplus F,E\oplus F)_0[2]^{\leq0}) &=-\chi(R\Hom(E\oplus F,E\oplus F)_0[2]^{<0})\\ &=\chi(R\Hom(E,F)_0)=\chi(E,F)=\chi(v,v_0) \end{align*}
Thus \(d=\chi(v,v_0)\), and it remains to prove
\[
\int_{\mathcal{X}\times\mathcal{Y}} (f,g)^{*}\,\widetilde{\Phi}_2^{*}\,\nu_{\Grad(\M^{ss})}
= \Bigl(\int_{\mathcal{X}} f^{*}\nu^{mot}_{\M^{ss}}\Bigr)\!
  \Bigl(\int_{\mathcal{Y}} g^{*}\nu^{mot}_{\M^{ss}}\Bigr).
\]

By the chosen orientation data (Theorem~\ref{orientationdata}), the induced orientation on \(\Grad(\M^{ss})\) pulls back to the product orientation on \(\M^{ss}\times \M^{ss}\) from the two factors, and the \(d\)-critical structures are compatible (Theorem~\ref{sss}). Hence
\[
\widetilde{\Phi}_2^{*}\,\nu^{mot}_{\Grad(\M^{ss})}
= \Phi_2^{*}\,\nu^{mot}_{\M^{ss}} \;=\; \nu^{mot}_{\M^{ss}\times \M^{ss}}.
\]
By the Thom--Sebastiani theorem (Theorem~\ref{TStheorem}),
\[
\nu^{mot}_{\M^{ss}\times \M^{ss}}
= \nu^{mot}_{\M^{ss}} \boxtimes \nu^{mot}_{\M^{ss}}.
\]
Combining these identities yields
\[
\int_{\mathcal{X}\times\mathcal{Y}} (f,g)^{*}\,\widetilde{\Phi}_2^{*}\,\nu^{mot}_{\Grad(\M^{ss})}
= \int_{\mathcal{X}\times\mathcal{Y}} (f,g)^{*}\bigl(\nu^{mot}_{\M^{ss,v}}\boxtimes \nu^{mot}_{\M^{ss,v_0}}\bigr)
= \Bigl(\int_{\mathcal{X}} f^{*}\nu^{mot}_{\M^{ss,v}}\Bigr)\!
  \Bigl(\int_{\mathcal{Y}} g^{*}\nu^{mot}_{\M^{ss,v_0}}\Bigr).
\]
Therefore \(I(a*b)=I(a)*I(b)\). The same argument shows \(I(b*a)=I(b)*I(a)\), which completes the proof.

\end{proof}

\subsection{DT/PT correspondence}
In order to make the statement of the DT/PT correspondence precise. We need to consider a completion of the Hall module $\mathcal{H}(\A_X)$. 
\begin{lemma}[{\cite[Lemma]{toda2020hall}}]
    We have the following observations.
    \begin{enumerate}
        \item For any $v\in\Gamma^0$, there is only finitely many ways to write $v=v_1+\cdots+v_\ell$ for $v_i\in\Gamma^0\setminus\{0\}$.
        \item For any $v\in\Gamma^1$ there is only finitely many ways to write $v=v_1\cdots+v_\ell+v_{\ell+1}$ for $v_1,\dots,v_\ell\in\Gamma^0\setminus\{0\}$ and $v_{\ell+1}\in\Gamma^1$.
    \end{enumerate}
\end{lemma}
For $*=0$ or $1$, we set
\[
\hat{\mathcal{H}}^*=\prod_{v\in\Gamma^*}\hat{\mathbb{M}}^{\mathrm{mon}}(\O bj^v(\mathcal{A}_X)),\qquad \hat{\mathbb{M}}^{\mathrm{mon}}(\mathbb{C})[[\Gamma^*]]=\prod_{v\in\Gamma^*}\hat{\mathbb{M}}^{\mathrm{mon}}(\mathbb{C})\cdot x^v
\]
The $*$-product is well-defined on $\hat{\mathcal{H}}^*$ by the above lemma, which defines an $\hat{\mathcal{H}}^0$-bimodule structure on $\hat{\mathcal{H}}^1$ and the integration map also extends to this completion ring
\[
\hat{I}:\hat{\mathcal{H}}^*\to \hat{\mathbb{M}}^{\mathrm{mon}}(\mathbb{C})[[\Gamma^*]], (a^v)_v\mapsto\sum I(a^v).
\]
As before, we also consider the $\hat{\mathcal{H}}^0$-submodule $\hat{\mathcal{H}}^{ss}$ generated by elements that factor through the semistable locus $\M^{ss}$.

Moreover, we may extend Theorem \ref{homomorphism} to this completion ring
\begin{theorem}
The restriction map $\hat{I}:\hat{\mathcal{H}}^{ss}\to\hat{\mathbb{M}}^{\mathrm{mon}}(\mathbb{C})[[\Gamma^*]]$, is an $\hat{\mathcal{H}}^0$-module homomorphism. In other words, for any $a\in\hat{\mathcal{H}}^{ss}$, $b\in\hat{\mathcal{H}}^0$ we have
\[
\hat{I}(a)*\hat{I}(b)=\hat{I}(a*b),\qquad \hat{I}(b)*\hat{I}(a)=\hat{I}(b*a)
\]
\end{theorem}
\begin{proof}
    This follows directly from Theorem \ref{homomorphism}.
\end{proof}
We consider the following two elements in $\hat{\mathcal{H}}^{ss}$.
\[
\delta_{*}=\left([\M^{v}_{*}\to\O bj^{v}(\mathcal{A}_X)]\right)_{v}\in\hat{\mathcal{H}}^1
\]
for $*=\mathrm{DT, }\,\mathrm{PT}$.
\noindent By definition we get
\begin{align*}
    &\hat{I}(\delta_{\mathrm{DT}})=\frac{1}{\L^{\frac{1}{2}}-\L^{-\frac{1}{2}}}\sum_{n,\beta}\mathrm{DT}^{\mathrm{mot}}_{n,\beta}x^{(-n,-\beta,1)}\\
    &\hat{I}(\delta_{\mathrm{PT}})=\frac{1}{\L^{\frac{1}{2}}-\L^{-\frac{1}{2}}}\sum_{n,\beta} \mathrm{PT}^{\mathrm{mot}}_{n,\beta}x^{(-n,-\beta,1)}
\end{align*}
We formally define
\[
    \epsilon_\infty := \log(\delta_\infty),
\]
where the logarithm is given by the usual formal power series
\[
    \log(1+x)=\sum_{n\ge1}\frac{(-1)^{n-1}}{n}x^n,
\]
applied to $\delta_\infty$ in the completed Hall algebra $\widehat{\mathcal{H}}^{0}$, so that $\exp(\epsilon_\infty)=\delta_\infty$.

Applying the integration map on this element we obtain some motivic invariants which we denoted by $N^\mathrm{mot}_n$.
\[
\hat{I}(\epsilon_\infty)=:\frac{1}{\L^{\frac{1}{2}}-\L^{-\frac{1}{2}}}\sum_{n\geq0} N^\mathrm{mot}_n\cdot x^{(-n,0,0)}.
\]
\noindent These elements satisfy the following identity in $\hat{\mathcal{H}}^1$.
\begin{lemma}[{\cite[Lemma 3.16]{toda2020hall}}]\label{DTPTwallcrossing}
    We have the following identity
    \[
    \delta_\mathrm{DT}*\delta_\infty=\delta_\infty*\delta_\mathrm{PT}\in\hat{\mathcal{H}}^1.
    \]
\end{lemma}
\begin{remark}
The factor $(\mathbb{L}^{1/2} - \mathbb{L}^{-1/2})$ appearing in the definition of motivic DT/PT invariants carries two important meanings. First, it may be written as $\mathbb{L}^{-1/2}(\mathbb{L} - 1)$, where the $\mathbb{L}^{-1/2}$ factor arises from the fact that we are integrating over the quotient stack $\mathcal{M}^v_{\mathrm{DT}} = [M^v_{\mathrm{DT}}/\mathbb{G}_m]$. The motivic Behrend function behaves compatibly with smooth morphisms, and in particular, the pullback of the motivic Behrend function along the smooth projection $M^v_{\mathrm{DT}} \to \mathcal{M}^v_{\mathrm{DT}}$ introduces a factor of $\mathbb{L}^{1/2}$. Hence, the integration over the quotient stack must be corrected by multiplying with $\mathbb{L}^{1/2}$.

Second, the factor $(\mathbb{L} - 1)$ ensures that the Euler characteristic of the resulting motivic generating series is well-defined. This is a key point in the so-called no pole theorem due to Joyce, which guarantees that the naive Euler characteristic of motivic DT invariants would otherwise diverge. The original proof of the no pole theorem is highly technical, but recent work by Kinjo, Bu, and Núñez \cite{IntrinsicDTII} has significantly simplified the argument and provided a more conceptual explanation of this phenomenon.
\end{remark}

We are now ready to prove the motivic DT/PT correspondence.
\begin{proof}[Proof of Theorem \ref{main}]
By Lemma \ref{DTPTwallcrossing} we have the following identity
    \[
    \delta_\mathrm{DT}=\delta_\infty*\delta_\mathrm{PT}*\delta_\infty^{-1}
    \]
    Now we define the Poisson bracket operator $\{\epsilon,x\}:=\epsilon*x-x*\epsilon$. Then we have
    \[
    \delta_\mathrm{DT}=\exp(\{\epsilon_\infty,-\})(\delta_\mathrm{PT})=\delta_\mathrm{PT}+\{\epsilon_\infty,\delta_{\mathrm{PT}}\}+\tfrac{1}{2!}\{\epsilon_\infty,\{\epsilon_\infty,\delta_{\mathrm{PT}}\}\}+\tfrac{1}{3!}\cdots
    \]
    Observed that for all $n\geq0$ and $v\in\Gamma^1$ we have $\chi((-n,0,0),v)=n$, thus
    \[
    x^{(-n,0,0)}*x^v-x^v* x^{(-n,0,0)}=\left(\left(\L^{\frac{n}{2}}-\L^{-\frac{n}{2}}\right)x^{(-n,0,0)}\right)\cdot x^{v}
    \]
    Thus, multiple both side by $\L^{\frac{1}{2}}-\L^{-\frac{1}{2}}$ and apply Theorem \ref{homomorphism} we obtained
    \begin{align*}
        \sum_{n,\beta}\mathrm{DT}^{\mathrm{mot}}_{n,\beta}x^{(-n,-\beta,1)}&=\exp\left(\sum_{n\geq0}\left(\frac{\L^{\frac{n}{2}}-\L^{-\frac{n}{2}}}{\L^{\frac{1}{2}}-\L^{-\frac{1}{2}}}\right)N_n^{\mathrm{mot}}x^{(-n,0,0)}\right)\cdot\left(\sum_{n,\beta}\mathrm{PT}^{\mathrm{mot}}_{n,\beta}x^{(-n,-\beta,1)}\right)
    \end{align*}
    denote $q=x^{-1,0,0}$ and $x^\beta=x^{(0,-\beta,1)}$ we may rewrite the above equation
    \[
    \sum_{n,\beta}\mathrm{DT}^{\mathrm{mot}}_{n,\beta}q^nx^\beta=\exp\left(\sum_{n\geq0}\left(\frac{\L^{\frac{n}{2}}-\L^{-\frac{n}{2}}}{\L^{\frac{1}{2}}-\L^{-\frac{1}{2}}}\right)N_n^{\mathrm{mot}}q^n\right)\left(\sum_{n,\beta}\mathrm{PT}^{\mathrm{mot}}_{n,\beta}q^nx^{\beta}\right)
    \]
    Since $\exp\left(\sum_{n\geq0}\left(\frac{\L^{\frac{n}{2}}-\L^{-\frac{n}{2}}}{\L^{\frac{1}{2}}-\L^{-\frac{1}{2}}}\right)N_n^{\mathrm{mot}}q^n\right)$ is independent of $\beta$ so comparing the coefficient we finally get
    \[\sum_{n,\beta}\mathrm{DT}^{\mathrm{mot}}_{n,\beta}q^nx^\beta=\left(\sum_{n\geq0}\mathrm{DT}^{\mathrm{mot}}_{n,0}q^n\right)\cdot\left(\sum_{n,\beta}\mathrm{PT}^{\mathrm{mot}}_{n,\beta}q^nx^{\beta}\right)
    \]
\end{proof}
\subsection{A refinement over the good moduli space}

The integration map can in fact be upgraded to the good moduli space; we now explain this. Let 
\[
\pi:\mathcal{M}^{ss,v}\longrightarrow M^{ss,v}
\]
be the good moduli space morphism for each numerical class 
$v\in\Gamma^{0}\cup\Gamma^{1}$, and set
\[
M^{ss}:=\coprod_{v\in\Gamma^{0}\cup\Gamma^{1}} M^{ss,v}.
\]

\begin{definition}[Motivic DT and PT functions]
    For each class $v \in \Gamma^{0} \cup \Gamma^{1}$, let
    \[
        \iota : \mathcal{M}^{v}_{\mathrm{DT}} \hookrightarrow \mathcal{M}^{ss,v}
    \]
    be the natural open immersion.
    We define the motivic Donaldson--Thomas function by
    \[
        \mathcal{DT}^{\mathrm{mot}}_{n,\beta}(X)
        := 
        \pi_{!}\,\iota_{!}\left(
            \bigl(
                \mathbb{L}^{\frac{1}{2}} - \mathbb{L}^{-\frac{1}{2}}
            \bigr)
            \nu^{\mathrm{mot}}_{\mathcal{M}^{v}_{\mathrm{DT}}}
        \right)
        \in
        \widehat{\mathbb{M}}^{\mathrm{mon}}(M^{ss,v}).
    \]
    The motivic Pandharipande--Thomas function
    $\mathcal{PT}^{\mathrm{mot}}_{n,\beta}(X)$ is defined similarly by
    replacing $\mathcal{M}^v_{\mathrm{DT}}$ with 
    $\mathcal{M}^v_{\mathrm{PT}}$.
\end{definition}

The direct sum morphism
\[
\Phi_2:\mathcal{M}^{ss}\times\mathcal{M}^{ss}\longrightarrow\mathcal{M}^{ss}
\]
descends, by the universal property of good moduli spaces,
to a unique morphism
\[
\oplus: M^{ss}\times M^{ss}\longrightarrow M^{ss}
\]
such that the diagram
\[
\begin{tikzcd}
	\mathcal{M}^{ss}\times\mathcal{M}^{ss}
	    \arrow[r,"\Phi_2"]
	    \arrow[d,"\pi\times\pi"'] &
	\mathcal{M}^{ss}
	    \arrow[d,"\pi"] \\
	M^{ss}\times M^{ss}
	    \arrow[r,"\oplus"] &
	M^{ss}
\end{tikzcd}
\]
commutes.

Using the morphism $\oplus$, we define convolution products on
motives over the good moduli space.  Set
\[
H^{ss}
	:=\bigoplus_{v\in\Gamma^{0}\cup\Gamma^{1}}
	\widehat{\mathbb{M}}^{\mathrm{mon}}(M^{ss,v}),
	\qquad
H^{0}
	:=\bigoplus_{v\in\Gamma^{0}}
	\widehat{\mathbb{M}}^{\mathrm{mon}}(M^{ss,v}).
\]
For $v\in\Gamma^{1}$ and $v_0\in\Gamma^{0}$, we define the twisted
convolution product
\[
*:
\widehat{\mathbb{M}}^{\mathrm{mon}}(M^{ss,v})
\times
\widehat{\mathbb{M}}^{\mathrm{mon}}(M^{ss,v_0})
\longrightarrow
\widehat{\mathbb{M}}^{\mathrm{mon}}(M^{ss,v+v_0})
\]
by
\[
(a,b)\ \longmapsto\ 
\mathbb{L}^{\chi(v,v_0)/2}\,\oplus_!\bigl(a\boxtimes b\bigr),
\]
and the untwisted convolution product
\[
\cdot:
\widehat{\mathbb{M}}^{\mathrm{mon}}(M^{ss,v})
\times
\widehat{\mathbb{M}}^{\mathrm{mon}}(M^{ss,v_0})
\longrightarrow
\widehat{\mathbb{M}}^{\mathrm{mon}}(M^{ss,v+v_0})
\]
by
\[
(a,b)\ \longmapsto\ \oplus_!\bigl(a\boxtimes b\bigr).
\]

We now define the integration map
\begin{align*}
    I_{\mathrm{GMS}}: \mathcal{H}^{ss} &\longrightarrow H^{ss}, \\
    \bigl[\mathcal{X} \xrightarrow{f} \mathcal{O}bj^v(\mathcal{A}_X)\bigr]
    &\longmapsto \pi_!f_!f^*\nu_{\mathcal{O}bj^{v}(\mathcal{A}_X)}^{\mathrm{mot}}.
\end{align*}

\begin{proposition}
    The twisted convolution product $*$ on $H^{ss}$ is compatible with the
    Hall algebra product on $\mathcal{H}^{ss}$ under the morphism
    $I_{\mathrm{GMS}}$.
\end{proposition}

\begin{proof}
    The proof is identical to that of Theorem~\ref{homomorphism},
    except that we push forward to the good moduli space $M^{ss}$ rather than
    to the point $\Spec(\mathbb{C})$.
\end{proof}

We consider the completions 
\[
\widehat{H}^{ss}
	:=\prod_{v\in\Gamma^{0}\cup\Gamma^{1}}
	\widehat{\mathbb{M}}^{\mathrm{mon}}(M^{ss,v}),
	\qquad
\widehat{H}^{0}
	:=\prod_{v\in\Gamma^{0}}
	\widehat{\mathbb{M}}^{\mathrm{mon}}(M^{ss,v}).
\]

\begin{corollary}\label{cor:I-GMS-hom}
    There is an induced morphism
    \[
        \widehat{I}_\mathrm{GMS}:\widehat{\mathcal{H}}^{ss}\longrightarrow\widehat{H}^{ss}
    \]
    which is a homomorphism of algebras with respect to the Hall product
    on $\widehat{\mathcal{H}}^{ss}$ and the twisted convolution product
    $*$ on $\widehat{H}^{ss}$.
\end{corollary}

Using the morphism $\widehat{I}_\mathrm{GMS}$ we define motivic functions
$\mathcal{N}^{\mathrm{mot}}_n\in\widehat{\mathbb{M}}^{\mathrm{mon}}(M^{ss,v_0})$
for $v_0\in\Gamma^0$ by the equality
\[
\widehat{I}_\mathrm{GMS}(\epsilon_\infty)
    =\frac{1}{\L^{\frac{1}{2}}-\L^{-\frac{1}{2}}}
      \sum_{n\geq0} \mathcal{N}_n^{\mathrm{mot}}.
\]

\begin{theorem}\label{mian2}
    There is an equality of motivic generating series in
    $\widehat{H}^{ss}$:
    \[
        \sum_{n,\beta}\mathcal{DT}^{\mathrm{mot}}_{n,\beta}
        \;=\;
        \left(\sum_{n\geq0}\mathcal{DT}^{\mathrm{mot}}_{n,0}\right)
        \cdot
        \left(\sum_{n,\beta}\mathcal{PT}^{\mathrm{mot}}_{n,\beta}\right).
    \]
\end{theorem}

\begin{proof}
Recall the Hall algebra elements
\[
\delta_{\mathrm{DT}}
\;=\;
\sum_{v\in\Gamma^{1}}
\bigl[\mathcal{M}^{v}_{\mathrm{DT}}
   \hookrightarrow \mathcal{O}bj^{v}(\mathcal{A}_X)\bigr],
\qquad
\delta_{\mathrm{PT}}
\;=\;
\sum_{v\in\Gamma^{1}}
\bigl[\mathcal{M}^{v}_{\mathrm{PT}}
   \hookrightarrow \mathcal{O}bj^{v}(\mathcal{A}_X)\bigr].
\]
As in the proof of Theorem~\ref{main}, there is an identity in
$\widehat{\mathcal{H}}^{ss}$ of the form
\[
    \delta_{\mathrm{DT}}
    \;=\;
    \exp(\{\epsilon_{\infty},-\}) \circ \delta_{\mathrm{PT}},
\]
where $\{-,-\}$ denotes the Poisson bracket associated to the Hall
product and twisted convolution.
By Corollary~\ref{cor:I-GMS-hom}, the
completed integration map
\[
\widehat{I}_{\mathrm{GMS}} : \widehat{\mathcal{H}}^{ss}\longrightarrow\widehat{H}^{ss}
\]
is a homomorphism of algebras with respect to the Hall product on
$\widehat{\mathcal{H}}^{ss}$ and the twisted convolution product $*$ on
$\widehat{H}^{ss}$. Applying $\widehat{I}_{\mathrm{GMS}}$ we obtain
\[
    \widehat{I}_{\mathrm{GMS}}(\delta_{\mathrm{DT}})
    \;=\;
    \widehat{I}_{\mathrm{GMS}}(\delta_{\mathrm{PT}})
    +\left\{\widehat{I}_\mathrm{GMS}(\epsilon_\infty),
      \widehat{I}_{\mathrm{GMS}}(\delta_{\mathrm{PT}})\right\}
    +\left\{\widehat{I}_\mathrm{GMS}(\epsilon_\infty),
      \left\{\widehat{I}_\mathrm{GMS}(\epsilon_\infty),
        \widehat{I}_{\mathrm{GMS}}(\delta_{\mathrm{PT}})\right\}\right\}
    +\cdots.
\]
By the definition of the motivic DT and PT functions we
have
\[
    \widehat{I}_{\mathrm{GMS}}(\delta_{\mathrm{DT}})
    =
    \sum_{n,\beta}\mathcal{DT}^{\mathrm{mot}}_{n,\beta},
    \qquad
    \widehat{I}_{\mathrm{GMS}}(\delta_{\mathrm{PT}})
    =
    \sum_{n,\beta}\mathcal{PT}^{\mathrm{mot}}_{n,\beta}.
\]
Moreover, for $m\ge 0$ we have the identity
\[
\mathcal{N}_m^\mathrm{mot} * \mathcal{PT}_{n,\beta}^\mathrm{mot}
\;-\;
\mathcal{PT}_{n,\beta}^\mathrm{mot} * \mathcal{N}_m^\mathrm{mot}
=
\left(\L^{m/2}-\L^{-m/2}\right)
    \,\mathcal{N}_m^\mathrm{mot}\cdot\mathcal{PT}_{n,\beta}^\mathrm{mot},
\]
so that the Poisson bracket $\{-,-\}$ can be expressed in terms of the
untwisted product $\cdot$ and the classes $\mathcal{N}_m^{\mathrm{mot}}$.
A standard computation (as in the proof of
Theorem~\ref{main}) then yields
\begin{align*}
    \sum_{n,\beta}\mathcal{DT}^{\mathrm{mot}}_{n,\beta}
    &=
    \exp\left(
        \sum_{n\geq0}\left(
            \frac{\L^{\frac{n}{2}}-\L^{-\frac{n}{2}}}
                 {\L^{\frac{1}{2}}-\L^{-\frac{1}{2}}}
        \right)\mathcal{N}_n^{\mathrm{mot}}
    \right)
    \cdot\left(\sum_{n,\beta}\mathcal{PT}^{\mathrm{mot}}_{n,\beta}\right).
\end{align*}

As in the proof of Theorem~\ref{main}, by looking at the degree
$\beta=0$ part we obtain
\[
\exp\left(
    \sum_{n\geq0}\left(
        \frac{\L^{\frac{n}{2}}-\L^{-\frac{n}{2}}}
             {\L^{\frac{1}{2}}-\L^{-\frac{1}{2}}}
    \right)\mathcal{N}_n^{\mathrm{mot}}
\right)
\;=\;
\sum_{n\ge0}\mathcal{DT}^{\mathrm{mot}}_{n,0}
\quad\in \widehat{H}^{0}.
\]
Substituting this back into the previous identity, we conclude that
\[
    \sum_{n,\beta}\mathcal{DT}^{\mathrm{mot}}_{n,\beta}
    \;=\;
    \left(\sum_{n\ge0}\mathcal{DT}^{\mathrm{mot}}_{n,0}\right)
    \cdot\left(\sum_{n,\beta}\mathcal{PT}^{\mathrm{mot}}_{n,\beta}\right)
    \quad\text{in }\widehat{H}^{ss},
\]
which is the desired statement.
\end{proof}

\subsection{Motivic invariants for moduli spaces of zero-dimensional sheaves}
The invariants $N_n:=\chi(N_n^\mathrm{mot})$ are known as the generalized Donaldson--Thomas invariants \cite{JoyceSong}. It was shown to satisfies the following identity
\cite[Section~6.3]{JoyceSong}:
\[
N_n = \sum_{k \mid n} \frac{-\chi(X)}{k^2}.
\]
We now compute these invariants at the motivic level using the ideas of \cite[Section 6.7]{DavisonMeinhardt2015}. In the context of the motivic DT/PT correspondence, the invariants \(N_n^{\mathrm{mot}}\) satisfy the identity
\[
\exp\left(\sum_{n \geq 0} \left(\frac{\mathbb{L}^{\frac{n}{2}} - \mathbb{L}^{-\frac{n}{2}}}{\mathbb{L}^{\frac{1}{2}} - \mathbb{L}^{-\frac{1}{2}}} \right) N_n^{\mathrm{mot}} q^n \right) = \sum_{n \geq 0} \mathrm{DT}^{\mathrm{mot}}_{n,0} q^n.
\]

The motivic DT invariants for curve class \(\beta = 0\) have already been studied in \cite{BehrendBryanSzendroi2013}. However, it is not clear whether the orientation chosen in \cite{BehrendBryanSzendroi2013} and the one constructed by \cite{JoyceUpmeier2021} agree. 
\begin{conjecture}\label{sameorientation}
The orientation constructed in \cite{JoyceUpmeier2021} agrees with the orientation used in \cite{BehrendBryanSzendroi2013}. In particular, the motivic Donaldson--Thomas invariants defined in this paper agree with those defined in \cite{BehrendBryanSzendroi2013}.
\end{conjecture}

\begin{theorem}[{\cite[Theorem 3.3]{BehrendBryanSzendroi2013}}]\label{degree0dt}
    Let $X$ be a smooth projective 
    Calabi--Yau threefold. If Conjecture \ref{sameorientation} holds, then
    \[
    \sum_{n\geq0}\mathrm{DT}^{\mathrm{mot}}_{n,0}(-q)^n=\mathrm{Exp}\left(\frac{-\L^{-\frac{3}{2}}[X]q}{(1+\L^{\frac{1}{2}}q)(1+\L^{-\frac{1}{2}}q)}\right).
    \]
\end{theorem}
Here, \(\mathrm{Exp}\) denotes the power structure exponential, defined by
\[
\sum_{n \geq 1} a_n q^n \mapsto \prod_{n \geq 1} (1 - q^n)^{-a_n},
\]
where the identity \((1 - q)^{-[X]} = \sum_{n \geq 0} [\mathrm{Sym}^n X] q^n\) holds in \(\hat{\mathbb{M}}^{\mathrm{mon}}(\mathbb{C})\). We refer to \cite{BehrendBryanSzendroi2013, GuseinZadeLuengoMelle2006} for precise definitions of this operator. We now consider the Adams operations \(\psi^k\) as described in \cite[Appendix B]{DavisonMeinhardt2015}, defined by the generating series
\[
\psi_q(a) := \sum_{k \geq 1} \psi^k(a) q^k = \frac{q \cdot \frac{d}{dq} (1 - q)^{-[X]}}{(1 - q)^{-[X]}}.
\]

Then we get
\[
(1-q)^{-[X]}=\exp\left(\int\psi_q([X])\frac{dq}{q}\right)=\exp\left(\sum_{k\geq1}\frac{1}{k}\psi^k([X])q^k\right).
\]
The operators $\psi^k$ satisfies some basic properties.
\begin{lemma}\label{k=nk}
    For a variety $X$, we have
    \[
    \psi^k((-\L^{\frac{1}{2}})^n[X])=(-\L^{\frac{1}{2}})^{nk}\psi^k([X])
    \]
\end{lemma}
\begin{proof}
    We use the following fact (see \cite[Section 1]{BehrendBryanSzendroi2013})
    \[
    (1-q)^{-(-\L^{\frac{1}{2}})^n[X]}=(1-((-\L^{\frac{1}{2}})^nq))^{[X]}.
    \]
    Then we get
    \[
    \exp\left(\sum_{k\geq1}\frac{1}{k}\psi^k((-\L^{\frac{1}{2}})^n[X])q^k\right)=\exp\left(\sum_{k\geq1}\frac{1}{k}\psi^k([X])((-\L^{\frac{1}{2}})^nq)^k\right).
    \]
    Taking the logarithm of both sides yields the desired formula.
\end{proof}
\begin{lemma}\label{euler}
    For a variety $X$, we have
    \[
    \chi(\psi^k([X]))=\chi(X),
    \]
    where $\chi(X)$ is the topological Euler characteristic of $X$.
\end{lemma}
\begin{proof}
    This follows from the fact that
    \[
    \chi((1-q)^{-[X]})=(1-q)^{-\chi(X)}.
    \]
    Indeed, taking logarithm of both side and use the fact that taking Euler characteristic is a ring homomorphism we get
    \[
    \sum_{k\geq1}\frac{1}{k}\chi(\psi^k([X]))q^k=-\chi(X)\log(1-q)=\sum_{k\geq1}\frac{1}{k}\chi(X)q^k.
    \]
    Comparing the coefficient yields the desired formula.
\end{proof}
\begin{lemma}\label{Exp=exp}
    If we have an identity 
    \[
    \exp\left(\sum_{n\geq1}a_n q^n\right)=\mathrm{Exp}\left(\sum_{m\geq1}b_m q^m\right),
    \]
    then
    \[
    a_n=\sum_{k|n}\frac{1}{k}\psi^k(b_{\frac{n}{k}}).
    \]
\end{lemma}
\begin{proof}
    By definition, the left hand side may be rewritten as
    \begin{align*}
            \mathrm{Exp}\left(\sum_{m\geq1}b_m q^m\right)
            &=\prod_{m\geq1}(1-q^m)^{-b_m}\\
            &=\prod_{m\geq1}\exp\left(\sum_{k\geq1}\frac{1}{k}\psi^k(b_m)q^{km}\right)\\
            &=\exp\left(\sum_{m\geq1}\sum_{k\geq1}\frac{1}{k}\psi^k(b_m)q^{km}\right)\\
            &=\exp\left(\sum_{n\geq1}\sum_{k|n}\frac{1}{k}\psi^k(b_{\frac{n}{k}})q^{n}\right)
    \end{align*}
Taking the logarithm of both sides yields the desired formula
\[
a_n=\sum_{k|n}\frac{1}{k}\psi^k(b_{\frac{n}{k}}).
\]
\end{proof}
Using Theorem~\ref{degree0dt} and Lemma~\ref{Exp=exp}, \ref{k=nk}, we may now derive an explicit formula for \(N_n^{\mathrm{mot}}\).
\begin{theorem}If conjecture \ref{sameorientation} holds,
    for any $n\geq0$ we have
    \[
    N_n^{\mathrm{mot}}=\sum_{k|n}\frac{(-1)^{k-1}}{k}\cdot\frac{\L^{\frac{1}{2}}-\L^{-\frac{1}{2}}}{\L^\frac{k}{2}-\L^{-\frac{k}{2}}}\cdot\psi^k(\L^{-\frac{3}{2}}[X]).
    \]
\end{theorem}
\begin{proof}
    We first calculate the power series
\begin{align*}
    \frac{-\L^{-\frac{3}{2}}[X]q}{(1+\L^{\frac{1}{2}}q)(1+\L^{-\frac{1}{2}}q)}
\end{align*}
Let $a:=-\L^{\frac{1}{2}}$. Then 
\begin{align*}
    \frac{-\L^{-\frac{3}{2}}[X]q}{(1+\L^{\frac{1}{2}}q)(1+\L^{-\frac{1}{2}}q)}
    &=\frac{a^{-3}[X]q}{(1-aq)(1-a^{-1}q)}\\
    &=a^{-3}[X]q(1+aq+a^2q^2+\cdots)(1+a^{-1}q+a^{-2}q^2+\cdots)\\
    &=a^{-3}[X]q(1+(a^{-1})(1+a^2)q+(a^{-1})^2(1+a^2+a^4)q^2+\cdots)\\
    &=\sum_{n\geq1}a^{-n-2}[X]\left(\sum_{k=0}^{n-1}(a^2)^k\right)q^n\\
    &=\sum_{n\geq1}a^{-n-2}[X]\frac{1-a^{2n}}{1-a^2}q^n\\
    &=\sum_{n\geq1}a^{-3}[X]\frac{a^{n}-a^{-n}}{a-a^{-1}}q^n\\
\end{align*}
By Theorem \ref{degree0dt}, we have
\[
\exp\left(\sum_{n \geq 0} \left(\frac{(-a)^n - (-a)^{-n}}{(-a)^ - (-a)^{-1}} \right) N_n^{\mathrm{mot}}(-q)^n \right) = \mathrm{Exp}\left(\sum_{n\geq1}a^{-3}[X]\frac{a^{n}-a^{-n}}{a-a^{-1}}q^n\right)
\]
Thus by Lemma \ref{Exp=exp},
\begin{align*}
    N_n^{\mathrm{mot}}
    &=-\frac{a-a^{-1}}{a^n-a^{-n}}\sum_{k|n}\frac{1}{k}\psi^k\left(a^{-3}[X]\frac{a^{\frac{n}{k}}-a^{-\frac{n}{k}}}{a-a^{-1}}\right)\\
    &=\sum_{k|n}\frac{1}{k}\cdot\frac{a-a^{-1}}{a^k-a^{-k}}\cdot\psi^k(-a^{-3}[X])\\
    &=\sum_{k|n}\frac{(-1)^{k-1}}{k}\cdot\frac{\L^{\frac{1}{2}}-\L^{-\frac{1}{2}}}{\L^{\frac{k}{2}}-\L^{-\frac{k}{2}}}\psi^k(\L^{-\frac{3}{2}}[X]).
\end{align*}
\end{proof}

\begin{remark}
Taking the Euler characteristic of both sides, using Lemma \ref{euler} we recover the following identity from \cite[Section~6.3]{JoyceSong}:
\[
N_n = \sum_{k \mid n} \frac{-\chi(X)}{k^2}.
\]
\end{remark}

\printbibliography

\end{document}